\documentclass[a4paper,11pt]{amsart}
\usepackage{amsmath,amsfonts,amsthm,amssymb,enumerate,ifpdf}
\usepackage{graphicx}

\numberwithin{equation}{section}

\newtheorem{theorem}{Theorem}[section]

\newtheorem{proposition}[theorem]{Proposition}

\newtheorem{lemma}[theorem]{Lemma}

{
\theoremstyle{definition}

\newtheorem{remark}[theorem]{Remark}

}

\title{Branched covering surface-knots with degree three have the simplifying numbers less than three}
\author{Inasa Nakamura}
\address{Faculty of Electrical, Information and Communication Engineering, Institute of Science and Engineering, Kanazawa University, \newline
Kakumamachi, Kanazawa, 920-1192, Japan}
\email{inasa@se.kanazawa-u.ac.jp}
\subjclass[2010]{Primary 57Q45; Secondary 57Q35}
\keywords{surface-knot; 2-dimensional braid; chart; 1-handle}

\begin{document}

\begin{abstract}
 
A branched covering surface-knot is a surface-knot in the form of a branched covering over a surface-knot. For a branched covering surface-knot, we have a numerical invariant called the simplifying number. We show that branched covering surface-knots with degree three have the simplifying numbers less than three. 
\end{abstract}

\maketitle

\section{Introduction}\label{sec1}

A chart \cite{CKS, Kamada02} is an oriented and labeled graph satisfying certain conditions, and an oriented surface-knot in 4-space is described by a chart on a disk. Let $n$ be a positive integer. 
It is known that for a chart $\Gamma$ on a disk of degree $n \leq 3$, $\Gamma$ is equivalent to a \lq\lq ribbon chart'', which consists of several \lq\lq free edges'' and edges connected with no vertices \cite{Kamada92}. It follows that the unknotting number of $\Gamma$, denoted by $u(\Gamma)\in \mathbb{Z}_{\geq 0}$, which is the minimum number of free edges necessary to add to deform $\Gamma$ to a union of free edges, satisfies $u(\Gamma)\leq 2$ \cite{Kamada99}. 
In this paper, we show a similar result for the simplifying number of  a branched covering surface-knot. A branched covering surface-knot is a surface-knot presented by a chart $\Gamma$ on a diagram of a surface-knot $F$, denoted by $(F, \Gamma)$. The simplifying number of $(F, \Gamma)$, denoted by $u(F,\Gamma) \in \mathbb{Z}_{\geq 0}$, is the minimum number of \lq\lq 1-handles with chart loops" necessary to add to deform $(F, \Gamma)$ to a \lq\lq simplified form" \cite{N5, N6}. 
Our main result is as follows. 
\begin{theorem}\label{thm1}
Let $n$ be a positive integer. Let $(F, \Gamma)$ be a branched covering surface-knot of degree $n$ such that $F$ is connected. 
If $n \leq 3$, then $u(F, \Gamma)\leq 2$. 
\end{theorem}

A branched covering surface-knot is a surface-knot in the form of a branched covering over a surface-knot. A {\it surface-knot} is a closed surface embedded smoothly into the Euclidean 4-space $\mathbb{R}^4$. We assume that surface-knots are oriented. 
Let $I \times I$ be a 2-disk for an interval $I$. 
For a surface-knot $F$, let $N(F)=I \times I \times F$ be a tubular neighborhood of $F$ in $\mathbb{R}^4$. A closed surface $S$ embedded in $N(F)$ is called a {\it branched covering surface-knot over $F$ of degree $n$} if it satisfies the following two conditions. 

\begin{enumerate}[(1)]
\item
The restriction $p|_{S} \,:\, S \rightarrow F$ is a branched covering map of degree $n$, where $p\,:\, N(F) \to F$ is the natural projection. 

\item The number of points consisting $S \cap p^{-1}(x)$ is $n$ or $n-1$ for any point $x \in F$.
\end{enumerate}
Take a base point $x_0$ of $F$. 
We say that two branched covering surface-knots over $F$ of degree $n$ are {\it equivalent} if there is a fiber-preserving ambient isotopy of $N(F)=I \times I \times F$ rel $p^{-1}(x_0)$ which carries one to the other. 
In this paper, we assume that the base surface-knot $F$ is connected. 

We consider a branched covering surface-knot $S$ of degree 2, as the simplest non-trivial case. We explain the chart presentation. 
Consider the singular set $\mathrm{Sing}(p_1(S))$ of the image of $S$ by the projection $p_1$ to $I \times N(F)$. Perturbing $S$ if necessary, we assume that $\mathrm{Sing}(p_1(S))$ consists of double point curves and isolated branch points. The image of $\mathrm{Sing}(p_1(S))$ by the projection to $F$ forms a finite graph $\Gamma$ on $F$ such that the degree of a vertex of $\Gamma$ is $1$. A chart is a graph obtained from $\Gamma$ by labelling each edge by the number $1$ and assigning an orientation to each edge by a certain rule. See Section \ref{sec2} for details. 

We explain the simplifying operation. 
Let $B^2$ be a unit 2-disk and let $I=[0,1]$. 
A {\it 1-handle} is a 3-ball $h=B^2 \times I$ smoothly embedded in $\mathbb{R}^4$ such that $h \cap F=(B^2 \times \partial I) \cap F$.  
The surface-knot obtained from $F$ by a {\it 1-handle addition} along $h$ is the surface
\[
(F-(\mathrm{Int} B^2 \times \partial I)) \cup (\partial B^2 \times I),
\]
which is denoted by $F+h$. We assume that $F+h$ is orientable, and we give $F+h$ the orientation induced from that of $F$. 
For a 1-handle $h=B^2 \times I$, we call $B^2 \times \{0\}$ and $B^2 \times \{1\}$ the {\it ends} of $h$. We take the {\it oriented core} (or simply the {\it core}) as an oriented path $\rho(t), t \in I$ with the orientation of $I$ in $\partial B^2 \times I \subset h$ such that $\rho(t) \in \partial B^2 \times \{t\}$ for $t \in I$. When both ends of $h$ are on a 2-disk $E$ in $F$, we determine the {\it core loop} of $h$ as the oriented closed path in $F+h$ obtained from the oriented core $\rho(t), t \in [0,1]$ by connecting the initial point $\rho(0)$ and the terminal point $\rho(1)$ by a simple arc in $E$, and we assign the core loop the induced orientation. We say that a 1-handle $h$ attached to a surface-knot $F$ is {\it trivial} if there is a 3-ball $B^3$ satisfying $(h \cup F) \cap B^3=h$. 
 
Let $(F, \Gamma)$ be a branched covering surface-knot of degree 2. We call a connected segment of an edge an {\it arc}. We call an edge/arc of a chart a {\it chart edge/arc}, and we call a chart edge connected with no vertices a {\it chart loop} or simply a {\it loop}. A chart edge is called a {\it free edge} if its end points are a pair of vertices of degree 1. 
Now, the chart $\Gamma$ of degree 2 consists of several free edges and chart loops. We consider an equivalence deformation as the local exchange of two parallel horizontal chart arcs with opposite orientations into two parallel  vertical chart arcs with induced orientations. This deformation is called a CI-M2 move; see Section \ref{sec2} for precise definition and other equivalent modifications. When we have a 1-handle with a chart loop $h(\sigma_1, e)$ in a neighborhood of an arc of a chart loop, applying a CI-M2 move and sliding an end of the 1-handle,  $h(\sigma_1, e)$ eliminates the chart loop, as illustrated in the first row of Fig. \ref{fig1}. When the orientation of the chart loop is opposite, then we turn around the 1-handle to make the loop along the core loop have the opposite orientation, as illustrated in the second row of Fig. \ref{fig1}. Repeating this operation, by an addition of $h(\sigma_1, e)$, $(F, \Gamma)$ of degree 2 deforms to 
\begin{equation}\label{eq-1-1}
(F, \Gamma_0) +h(\sigma_1, e),
\end{equation}
where $\Gamma_0$ is a chart consisting of several (maybe no) free edges. 

\begin{figure}[ht]
\centering
\includegraphics*[height=6cm]{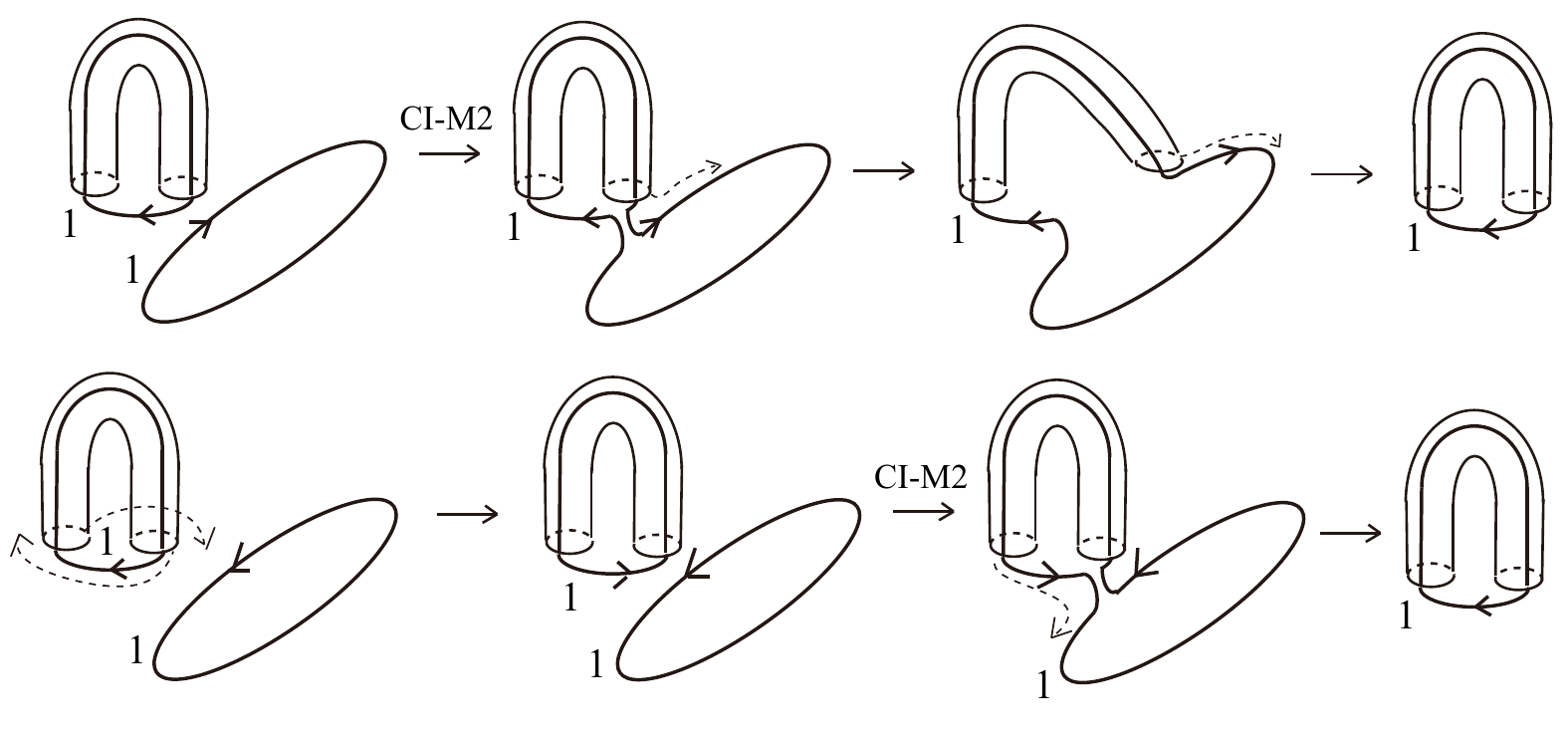}
\caption{The 1-handle with a chart loop $h(\sigma_1, e)$ eliminates a chart loop with the same label.}
\label{fig1}
\end{figure}

We call the form (\ref{eq-1-1}) a {\it simplified form}. Thus the simplifying number, denoted by $u(F, \Gamma)$, satisfies
\[
u(F, \Gamma) \leq 1
\]
for a branched covering surface-knot of degree 2. 
 
In this paper, we focus on branched covering surface-knots of degree 3. 
A {\it surface diagram} of a surface-knot $F$ is the image of $F$ in $\mathbb{R}^3$ by a generic projection, equipped with over/under information on sheets along each double point curve. 
A finite graph $\Gamma$ on a surface diagram $D$ is called a {\it chart} of degree $3$ if 
it satisfies 
the following conditions. 

\begin{enumerate}[(1)]
\item
The intersection of $\Gamma$ and the singular set of $D$ consists of a finite number of transverse intersection points of edges of $\Gamma$ and double point curves of $D$, which form vertices of degree $2$.

  \item Every edge of $\Gamma$ is oriented and labeled by an element of 
       $\{1,2\}$. 
  \item Every vertex has degree $1$, $2$, or $6$, and the edges connected to each vertex satisfy one of the following conditions.        
       \begin{enumerate}
       \item  The adjacent edges around a vertex of degree $1$ or $6$  
have labels and orientations as in Fig. \ref{fig2}.  
We depict the vertices as shown in Fig. \ref{fig2}, and we call a vertex of degree 1 a {\it black vertex}, and we call a vertex of degree 6 a {\it white vertex}.
\item
The adjacent edges of a vertex of degree 2 have labels and orientations as in Fig. \ref{fig3}. 
\end{enumerate}
 \end{enumerate}

\begin{figure}[ht] 
\includegraphics*{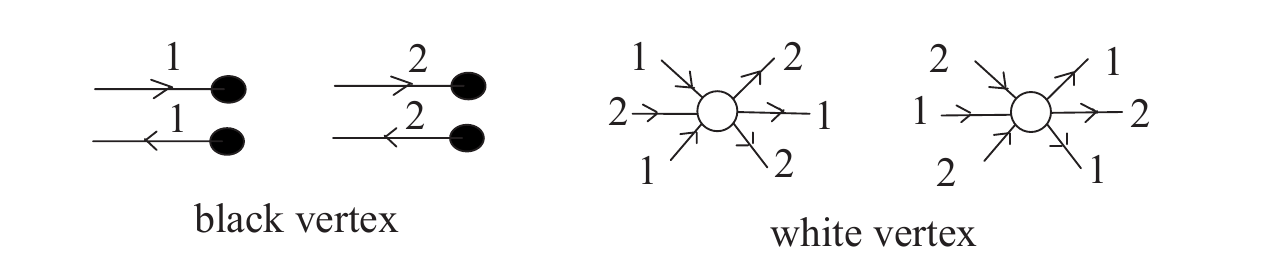}
\caption{Vertices of degree 1 or 6 in a chart of degree 3.}
\label{fig2}
\end{figure}

\begin{figure}[ht]
\centering
\includegraphics*[height=1.5cm]{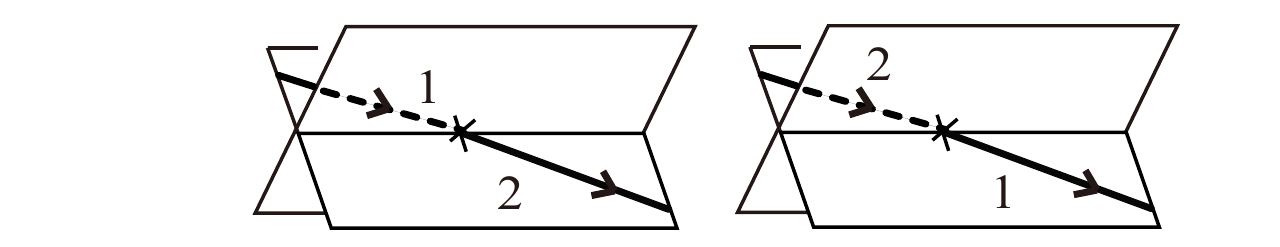}
\caption{A vertex of a degree 2 in a chart of degree 3. For simplicity, we omit the over/under information of each sheet.}
\label{fig3}
\end{figure}

 We regard edges connected with vertices of degree 2 as one edge. 
 Similarly to the case of charts of degree 2, a {\it free edge} is a chart edge whose end points are a pair of black vertices, and a {\it chart loop} (or simply a {\it loop}) is a chart edge connected with no vertices or a chart edge with vertices of degree 2. 
 A chart is said to be {\it empty} if it is an empty graph. 

 We consider a 1-handle on which a chart is drawn, attached to a fixed 2-disk $E$, where there are no chart edges nor vertices except those of the attached 1-handle. 
 We denote by $h(\sigma_i, e)$ $(i=1, 2)$ a 1-handle with a chart loop along the core loop with the orientation of the core loop and whose arc in $E$ has the label $i$. And we denote by $h(e,e)$ a 1-handle with an empty chart. We use the same notation $h$ for any 1-handle, and further, we do not distinguish the frame of a 1-handle. 

For a branched covering surface-knot $(F, \Gamma)$ and a finite number of 1-handles with chart loops  $h(\sigma_{i_k}, e)$ attached to a fixed 2-disk $E$, where there are no chart edges nor vertices except those of the attached 1-handles, we denote the branched covering surface-knot which is the result of the 1-handle addition by $(F, \Gamma) + \sum_{k} h(\sigma_{i_k},e)$. 
By \cite{N5}, for the degree 3 case, we have the following result. 
See \cite{N6} (see also Remark \ref{rem-2}) for the same notations and terminologies we use here. 

\begin{theorem}[{\cite[Theorem 1.6]{N5}}] \label{thm2}
Let $(F, \Gamma)$ be a branched covering surface-knots of degree 3. 
By an addition of finitely many 1-handles in the form $h(\sigma_1, e)$, $h(\sigma_2, e)$ or $h(e,e)$, to appropriate places in $F$, $(F, \Gamma)$ is deformed to 
\begin{equation}\label{eq1-1}
(F, \Gamma_0) +\sum_{k} h(\sigma_{i_k}, e)+\sum h(e,e),
\end{equation}
where $i_k \in \{1, 2\}$, and $\Gamma_0$ is a chart consisting of several (maybe no) free edges. 
\end{theorem}

We remark that the presentation (\ref{eq1-1}) is well-defined: Since $\Gamma_0$ is a disjoint union of free edges, (\ref{eq1-1}) does not depend on the place where 1-handles are attached. 
We call the form (\ref{eq1-1}) a {\it simplified form}. The {\it simplifying number} of $(F, \Gamma)$, denoted by $u(F, \Gamma)$, is the minimum number of 1-handles in the form $h(\sigma_1, e)$, $h(\sigma_2, e)$ or $h(e,e)$ necessary to deform $(F, \Gamma)$ to a simplified form. 
We further remark that in \cite{N5}, we gave two notions of the simplifying number: the weak simplifying number $u_w(F, \Gamma)$ and the simplifying number $u(F, \Gamma)$, but they are the same notion for the degree 3 case, where we have no vertices of degree 4. See \cite{N5,N7,N6} for the investigations of simplifying operations and the upper estimates of the (weak) simplifying numbers.  

Together with the above argument for the case of degree 2, Theorem \ref{thm1} is reduced to Theorem \ref{thm3} as follows. 

\begin{theorem}\label{thm3}
Let $(F, \Gamma)$ be a branched covering surface-knot of degree $3$. 
Then $u(F, \Gamma)\leq 2$. Further, if $\Gamma$ has a positive number of black vertices, then $u(F, \Gamma) \leq 1$. 
\end{theorem}

In order to show the latter part of Theorem \ref{thm3}, we show the following theorem. The case when $F$ is a 2-sphere is shown by Kamada \cite{Kamada92}. 

\begin{theorem}\label{thm4}
Let $(F, \Gamma)$ be a branched covering surface-knot of degree $3$ such that $\Gamma$ has a positive number of black vertices. Then, $\Gamma$ is equivalent to a chart consisting of free edges and chart loops. 
\end{theorem}

 The paper is organized as follows. In Section \ref{sec2}, we review chart presentations and their equivalent local modifications: C-moves and Roseman moves. 
In Section \ref{sec3}, we show Theorem \ref{thm2}. In Section \ref{sec4}, we show Theorems \ref{thm3} and \ref{thm4}. Section \ref{sec5} is devoted to showing Lemmas.

\section{Chart presentations, C-moves and Roseman moves}\label{sec2}
A branched covering surface-knot is a surface we formerly called a 2-dimensional braid over a surface-knot \cite{N4, N5}; we changed the terminology in \cite{N5}. We introduced a branched covering surface-knot as an extended notion of a 2-dimensional braid or a surface braid over a 2-disk \cite{Kamada92, Kamada02, Rudolph}. A branched covering surface-knot over a surface-knot $F$ is presented by a chart on a surface diagram of $F$ \cite{N4} (see also \cite{Kamada92, Kamada02}). For two branched covering surface-knots of the same degree, they are equivalent if their surface diagrams with charts are related by a finite sequence of ambient isotopies of $\mathbb{R}^3$, and local modifications called C-moves \cite{Kamada92, Kamada02} and Roseman moves \cite{N4} (see also \cite{Roseman}). In this paper, we review these notations for the case of degree 3. 
 
\subsection{Chart presentation}\label{sec2-2}

Let $S$ be a branched covering surface-knot of degree 3 over a surface-knot $F$. 
We explain how to obtain a chart of degree 3 on a 2-disk $E$ in a surface diagram $D$ of $F$ which does not intersect with singularities of $D$. 
We identify $E$ with a 2-disk $E \subset F$ whose projected image is $E \subset D$, and 
we denote $S\cap p^{-1}(E)$ by $S$, where $p$ is the projection $N(F) \to F$. We identify $N(E)$ by $I \times I \times E$ for an interval $I$. 
Consider the singular set $\mathrm{Sing}(p_1(S))$ of the image of $S$ by the projection $p_1$ to $I \times E$. Perturbing $S$ if necessary, we assume that $\mathrm{Sing}(p_1(S))$ consists of double point curves, isolated triple points, and isolated branch points. Thus 
the image of $\mathrm{Sing}(p_1(S))$ by the projection to $E$ forms a finite graph $\Gamma$ on $E$ such that the degree of a vertex of $\Gamma$ is either $1$ or $6$, where we ignore the points in $\partial E$. 
An edge of $\Gamma$ presents 
to a double point curve, and a vertex of degree $1$ and degree $6$ present a branch point and a triple point, respectively. 

For such a graph $\Gamma$ obtained from a branched covering surface $S$ of degree 3, we assign labels and orientations to all edges of $\Gamma$ by the following method. We consider a path $\rho$ in $E$ such that $\rho \cap \Gamma$ is a point $x$ of an edge $\alpha$ of $\Gamma$. Then $S \cap p^{-1} (\rho)$ is a $3$-braid with one crossing in the cylinder $p^{-1}(\rho)$ such that $x$ corresponds to the crossing of the braid. Let $\sigma_{i}^{\epsilon}$ ($i \in \{1,2\}$, 
$\epsilon \in \{+1, -1\}$) be the presentation. We assign the edge $\alpha$ the label $i$, and the orientation such that 
the normal vector of $\rho$ is coherent (respectively, is not coherent) with the orientation of $\alpha$ if $\epsilon=+1$ (respectively, $-1$), where the normal vector of $\rho$ is the vector $\vec{n}$ such that for a tangent vector $\vec{v}(\rho)$ of $\rho$ at $x$, $(\vec{v}(\rho), \vec{n})$ is coherent with the orientation of the 2-disk $E$. The resulting $\Gamma$ is the chart of degree 3 presenting $S$. 
Conversely, when we have a chart on a surface diagram, we can construct the branched covering surface-knot with the chart presentation. See \cite{N4} for the structure of a branched covering surface-knot for neighborhoods of the singularities of the surface diagram or vertices of degree 2 in a chart. 
\\
  
Around a white vertex of a chart, there are six chart arcs. We call the arc which is the middle of the three adjacent arcs with the coherent orientation a {\it middle arc}, and we call an arc which is not a middle arc a {\it non-middle arc}. Around a white vertex, there are two middle arcs and four non-middle arcs. For edges/arcs around a white vertex, we call a pair of edges/arcs separated by two edges at each side {\it diagonal edges/arcs}. See Fig. \ref{fig4}.

\begin{figure}[ht] 
\includegraphics*[height=3cm]{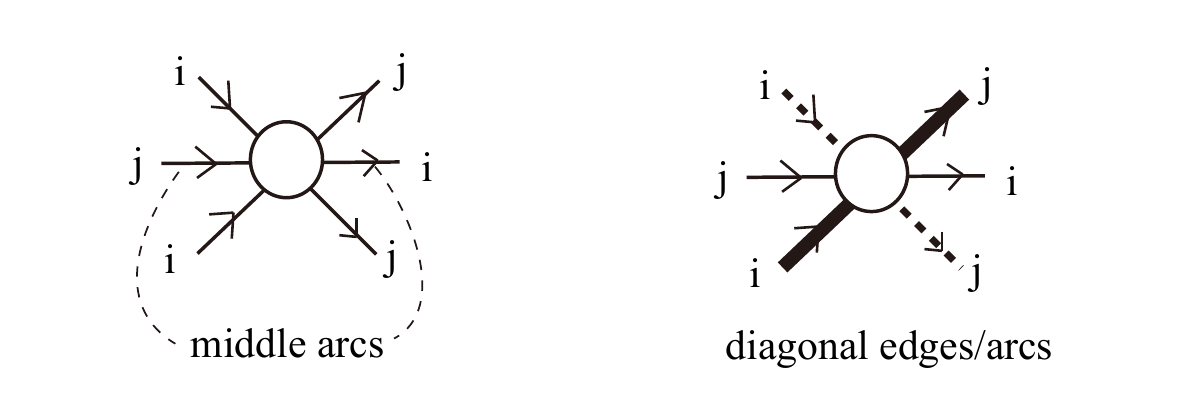}
\caption{Middle arcs and diagonal edges/arcs of a white vertex, where $\{i,j\}=\{1,2\}$.}
\label{fig4}
\end{figure}
 
\subsection{C-moves}
  
{\it C-moves} ({\it chart moves}) are local modifications of a chart, consisting of three types: CI-moves, CII-moves, and CIII-moves. CII-moves are modifications for charts of degree $n>3$, where we have vertices of degree 4 called crossings \cite{Kamada02}. We review CI-moves and CIII-moves for charts of degree 3. 
Let $\Gamma$ and $\Gamma^{\prime}$ be two charts of degree 3 on a surface diagram $D$. We say that $\Gamma$ and $\Gamma'$ are related by a {\it CI-move} or {\it CIII-move} if there exists a 2-disk $E$ in $D$ such that $E$ does not intersect with the singularities of $D$, and the loop $\partial E$ is in general position with respect to $\Gamma$ and $\Gamma^{\prime}$ and $\Gamma \cap (D-E)=\Gamma^{\prime} \cap (D-E)$, and the following conditions are satisfied.  
 
(CI) There are no black vertices in $\Gamma \cap E$ nor $\Gamma^{\prime} \cap E$. We use a CI-M1 move, a CI-M2 move and a CI-M3 move, as shown in Fig. \ref{fig5}. 
 
(CIII) $\Gamma \cap E$ and $\Gamma' \cap E$ are as in Fig. \ref{fig5}, where $\{i,j\}=\{1,2\}$, and the black vertex is connected to a non-middle arc of a white vertex. 
\\

\begin{figure}[ht]
\includegraphics*[width=13cm]{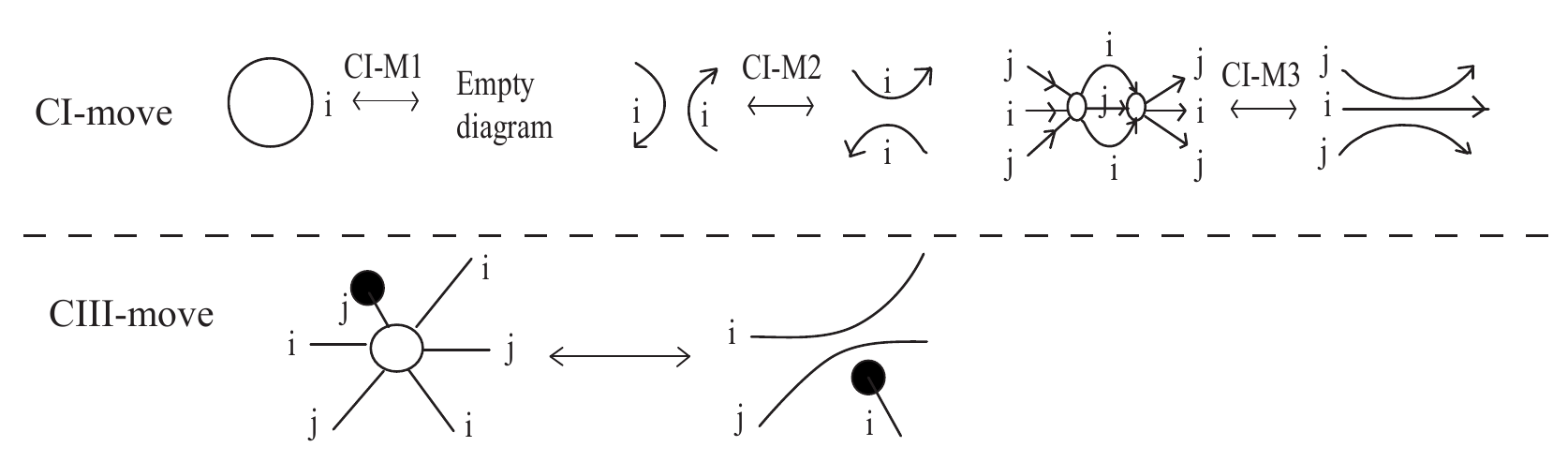}
\caption{C-moves, where $\{i,j\}=\{1,2\}$. For simplicity, we omit orientations of some of the edges.} 
\label{fig5}
\end{figure}

In this paper, we also use modifications of a CI-M3 move, called a {\it CI-M3' move} or a {\it CI-M3" move}, as shown in Fig. \ref{fig6}. We show in Fig. \ref{fig7} a proof of one of  CI-M3' moves as shown in the top left figure of Fig. \ref{fig6}. The other modifications in Fig. \ref{fig6} are shown similarly. 

\begin{figure}[ht]
\includegraphics*[width=13cm]{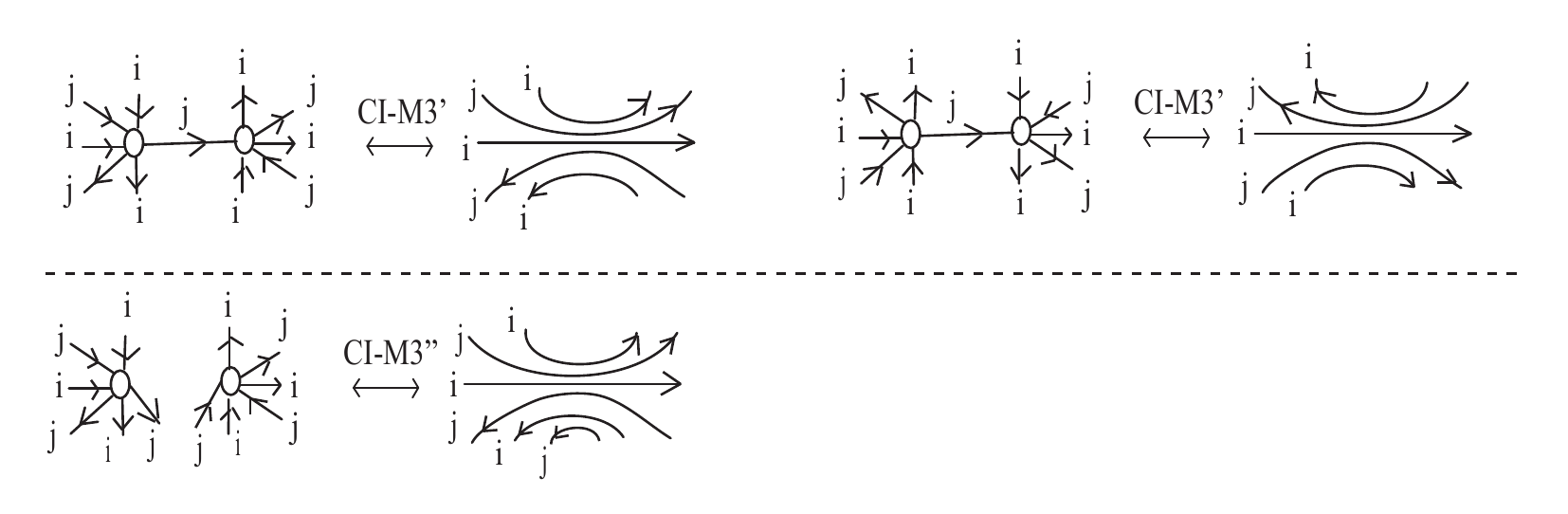}
\caption{CI-M3' moves and a CI-M3'' move, where $\{i,j\}=\{1,2\}$.}
\label{fig6}
\end{figure}

\begin{figure}[ht]
\includegraphics*[width=13cm]{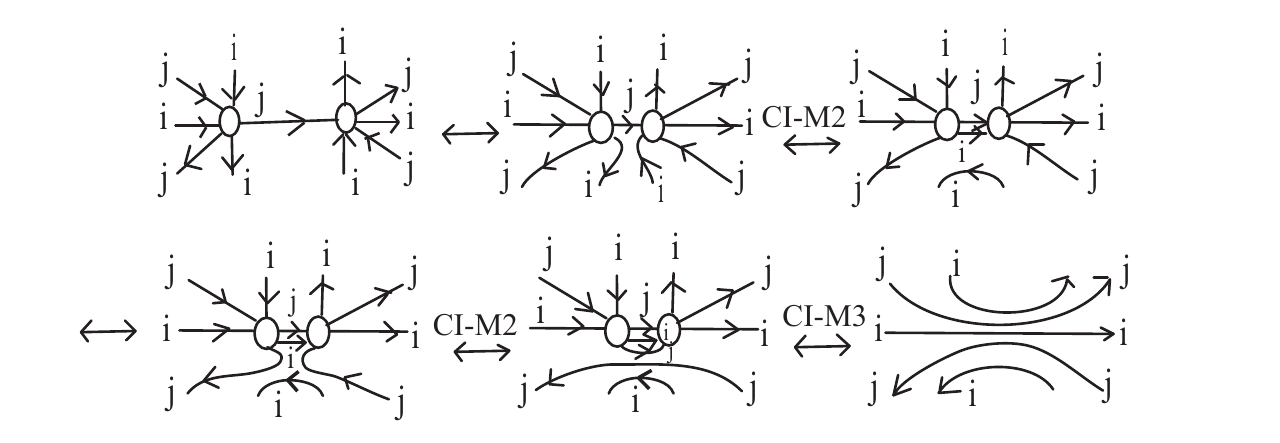}
\caption{Proof of a CI-M3' move, where $\{i,j\}=\{1,2\}$.} 
\label{fig7}
\end{figure}

\subsection{Roseman moves}\label{sec2-3}
 
{\it Roseman moves for surface diagrams with charts of degree 3} are defined by the original Roseman moves (see \cite{Roseman}) and moves for local surface diagrams with non-empty charts as in Fig. \ref{fig8} (see \cite{N4}), where we regard the diagrams for the original Roseman moves as equipped with empty charts. 
 
\begin{figure}[ht]
 \includegraphics*{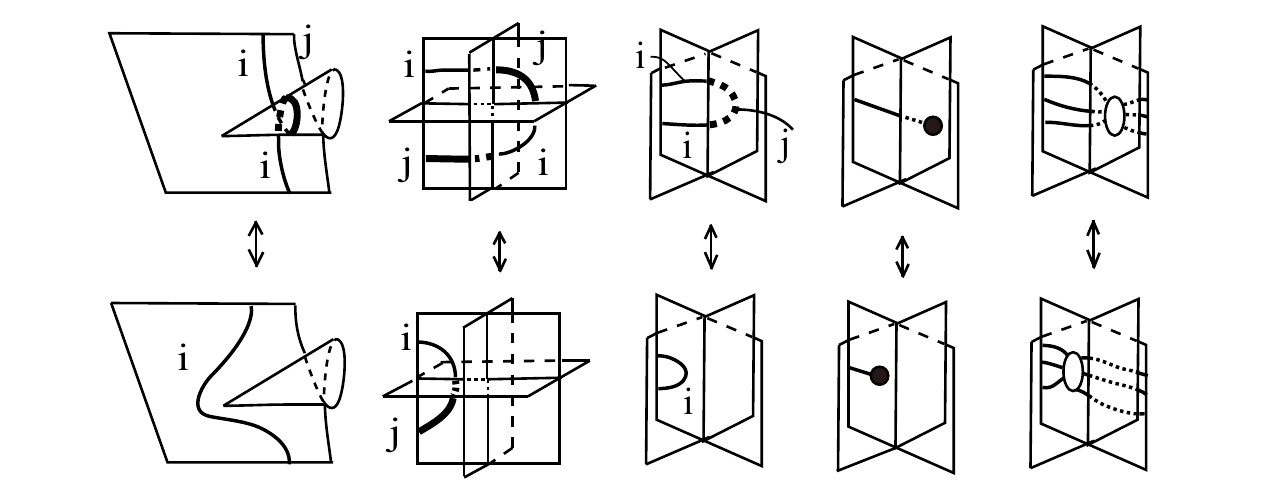}
\caption{Roseman moves for surface diagrams with charts of degree 3, where $\{i,j\} \in \{1,2\}$. For simplicity, we omit the over/under information of each sheet, and orientations and labels of chart edges. }
\label{fig8}
\end{figure}
 
For charts $\Gamma$ and $\Gamma'$ of degree 3 on a surface diagram of a surface-knot $F$, 
their presenting branched covering surface-knots are equivalent if the charts are related by a finite sequence of ambient isotopies of $\mathbb{R}^3$, C-moves and Roseman moves \cite{N4} (see also \cite{Kamada92, Kamada02, Roseman}).  

\begin{remark}\label{rem-2}
We remark that the deformations we use for the simplifying operation, such as sliding one end of a 1-handle, can be investigated by a similar method whether we have vertices of degree 2 or not. Hence, in the proofs of lemmas and theorems, we ignore the knottedness of $F$ and 1-handles. We assume that there are no singularities in surface diagrams of $F$ and 1-handles; in particular, we assume that 1-handles are trivial.  
But, one lemma, Lemma \ref{lem7}, requires the concerning 1-handle to be trivial. So we must be careful when we use Lemma \ref{lem7}; see Remark \ref{rem1}. 
\end{remark}

\section{Simplifying branched covering surface-knots}\label{sec3}
In this section, we show Theorem \ref{thm2}. We review lemmas used in \cite{N5}, and we explain the simplifying operation. Figures \ref{fig9} and \ref{fig10} are used in \cite{N5, N7, N6}. A figure similar to Fig. \ref{fig11} is in \cite{N5}. 

We consider branched covering surface-knots of degree 3. 
  
\begin{lemma}\label{lem1}
Let $\rho$ be a chart loop with the label $i$ $(i \in \{1,2\})$ in a branched covering surface-knot $(F, \Gamma)$. 
If there is a 1-handle $h(\sigma_i, e)$ near a neighborhood of an arc of $\rho$, then $(F, \Gamma)$ is equivalent to $(F, \Gamma\backslash \rho)$, where $\Gamma\backslash \rho$ denotes the chart obtained from $\Gamma$ by elimination of $\rho$. \end{lemma}

\begin{proof}
Applying a CI-M2 move between arcs of $h=h(\sigma_i,e)$ and sliding an end of $h$ as in Fig. \ref{fig1}, we eliminate $\rho$, and the other parts of $\Gamma$ remains unchanged. Thus we have the required result. 
\end{proof}

\begin{proposition}\label{prop1}
Let $(F, \Gamma)$ be a branched covering surface-knot of degree 3 such that $\Gamma$ has no white vertices. By an addition of $h(\sigma_1, e)+h(\sigma_2,e)$, $(F, \Gamma)$ deforms to a simplified form, hence $u(F, \Gamma) \leq 2$. 
\end{proposition}

\begin{proof}
Since $\Gamma$ has no white vertices, $\Gamma$ consists of several free edges and chart loops. By an addition of $h(\sigma_1, e)+h(\sigma_2, e)$, applying Lemma \ref{lem1} to each chart loop from the loop whose neighborhood has the 1-handles, we eliminate all chart loops. Thus we have a simplified form, and $u(F, \Gamma) \leq 2$.  
\end{proof}

\begin{lemma}\label{lem2}
If there is a 1-handle $h=h(\sigma_i, e)$ near a neighborhood of a non-middle arc with the label $i$ of a white vertex $w$, then, by equivalent deformation, $h$ collects $w$ as illustrated in Fig. \ref{fig9}.
\end{lemma}

\begin{figure}[ht]
\centering
\includegraphics*[width=13cm]{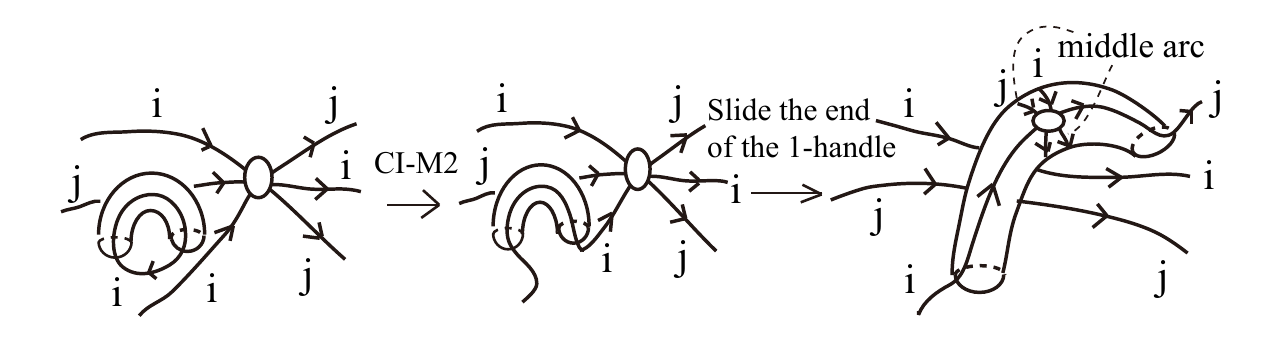}
\caption{Collecting a white vertex on a 1-handle, where $\{i,j\}=\{1,2\}$. See Fig. \ref{fig10} for deformations when we slide the end of the 1-handle. The orientations of the edges are an example.}
\label{fig9}
\end{figure}

\begin{proof}
Apply a CI-M2 move between the arc of $h=h(\sigma_i,e)$ and the non-middle arc of $w$. Move $w$ on $h$, and apply CI-M2 moves twice, first between arcs with the label $j$ and secondly between arcs with the label $i$, where $\{i,j\}=\{1,2\}$; see Fig. \ref{fig10}. Then $w$ is on $h$ such that there are two diagonal arcs along the core, and the other four arcs become two edges whose both endpoints are connected with $w$, as in the rightmost figure of Fig. \ref{fig9}. 
\end{proof}
\begin{figure}[ht]
\centering
\includegraphics*[width=13cm]{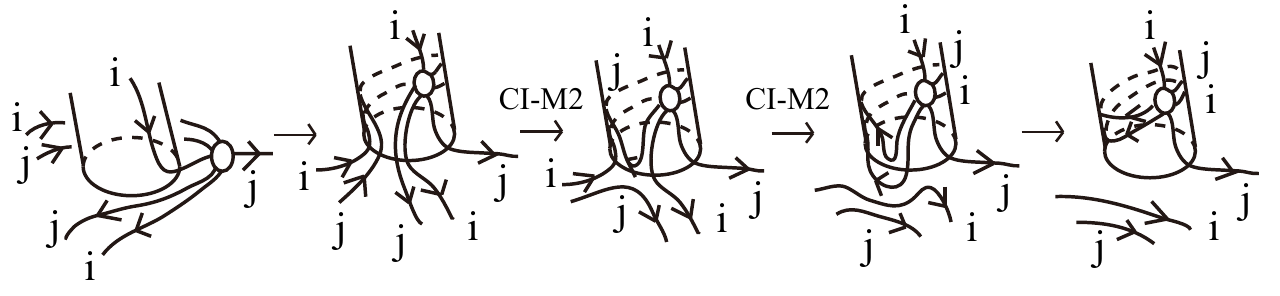}
\caption{Sliding the end of a 1-handle, where $\{i,j\}=\{1,2\}$. }
\label{fig10}
\end{figure}

\begin{lemma}\label{lem3}
Let $h$ be a 1-handle on which there is a white vertex $w$ such that two edges are at both endpoints connected with $w$, as in the rightmost figure of Fig. \ref{fig9}. Let $i,j$ $(\{i,j\}=\{1,2\})$ be the labels of edges. 
Let $\alpha$ be one of the edges whose both endpoints are connected with $w$, and let $k$ $(k=i,j)$ be the label of $\alpha$. 
If we add a 1-handle $h'=h(\sigma_k, e)$ in a neighborhood of $\alpha$, then, after collecting $w$ from $h$ to $h'$, the orientations of the two edges connected at both endpoints with $w$ are reversed from when they were on $h$, where we identify the oriented diagonal edges along the core of $h$ with those of $h'$. 
\end{lemma}

\begin{proof}
Note that for the two arcs of $\alpha$ connected with $w$, there is only one arc which is a non-middle arc, along which we slide an end of $h'$ and collect $w$. See Fig. \ref{fig11} for the case $k=i$. For the other case $k=j$, the form of $h'$ after collecting $w$ is the same with the case $k=i$, and we have the required result. 
\end{proof}

\begin{figure}[ht]
\centering
\includegraphics*[width=11cm]{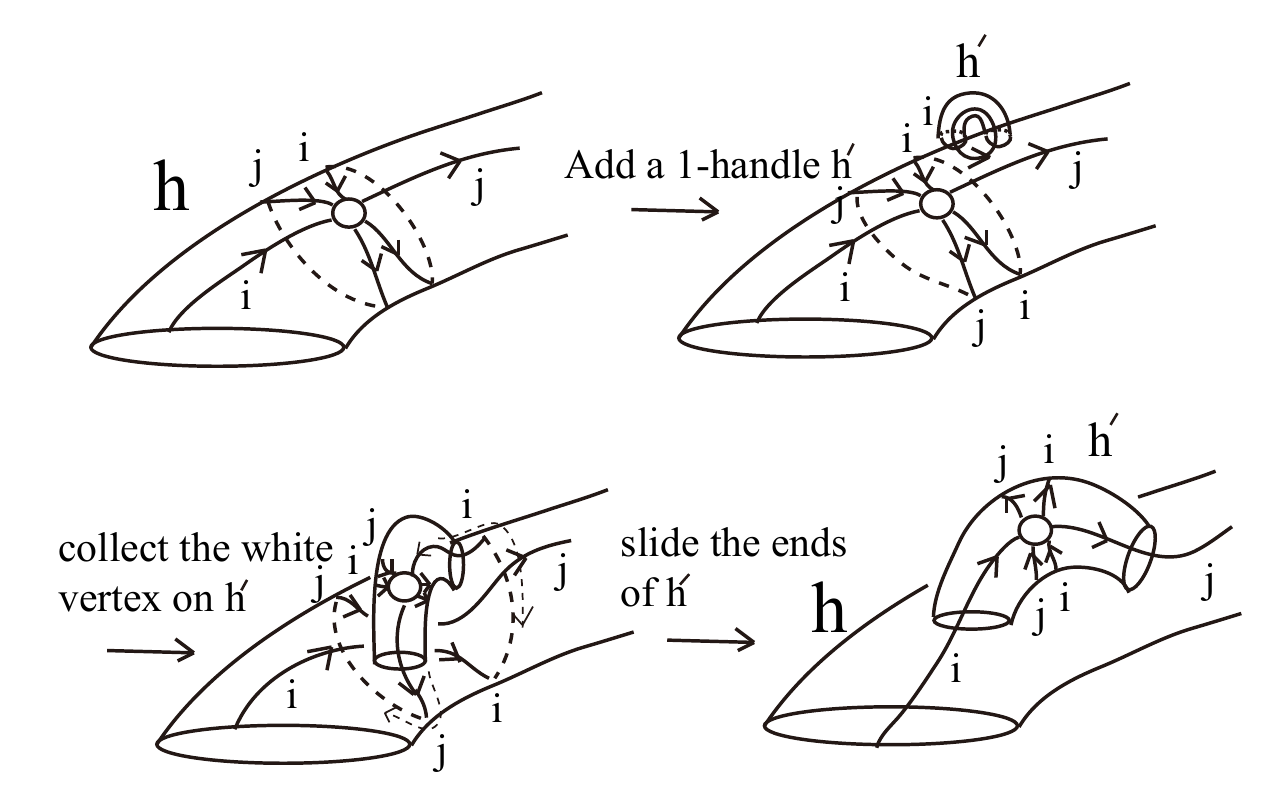}
\caption{Reversal of orientations of the edges connected at both endpoints with a white vertex on a 1-handle, where $\{i,j\}=\{1,2\}$. The orientations of the edges are an example.}
\label{fig11}
\end{figure}

\begin{proof}[Proof of Theorem \ref{thm2}]
If $\Gamma$ has no white vertices, then the result follows from Proposition \ref{prop1}. 
We assume that $\Gamma$ has a positive number of white vertices. 
We consider the case when $\Gamma$ has no black vertices. 
We add a 1-handle $h$ in a neighborhood of a non-middle arc of a white vertex $w_1$. Apply a CI-M2 move, slide an end of $h$, and by moves illustrated in Figs. \ref{fig9} and \ref{fig10}, collect $w_1$ on $h$. Then we slide the end of $h$. If $h$ comes to a non-middle arc of another white vertex $w_2$, then we collect $w_2$ on $h$. If $h$ comes to a middle arc of another white vertex $w_2$, then we add another 1-handle $h'$ near a non-middle arc of $w_2$, apply a CI-M2 move, slide an end of $h'$ and collect $w_2$ on $h'$, and let the end of $h$ pass under $h'$. Repeat this process, until the end of $h$ comes back to a neighborhood of the other end of $h$. The resulting $h$ has several white vertices. Repeat this process, until we have all white vertices on 1-handles. 

Take a 1-handle $h$ on which there are several white vertices, and 
consider edges connected to one white vertex at both endpoints.  If the orientations of these edges for one white vertex $w_1$ are opposite to those for another adjacent white vertex $w_2$, then, as shown in Fig. \ref{fig12}, by a CI-M3' move and two CI-M1 moves, we eliminate the pair of vertices $(w_1, w_2)$. If the orientations of the edges of $w_1$ and $w_2$ are coherent, then, as shown in Fig. \ref{fig11}, add a 1-handle $h'$ in a neighborhood of $w_1$, apply a CI-M2 move, slide the ends of $h'$ and collect $w_1$ and then $w_2$ and the other vertices on $h'$. By Lemma \ref{lem3}, on $h'$, the orientations of the edges whose endpoints are connected with $w_1$ are reversed. By the same process with the former case, by modifications as shown in Fig. \ref{fig12}, we eliminate $(w_1, w_2)$. The resulting $h'$ is a 1-handle obtained from $h$ by eliminating $(w_1, w_2)$. Repeat this process, until we eliminate all white vertices. The resulting chart consists of several chart loops and 1-handles with chart loops. Adding $h(\sigma_i, e)$ ($i \in \{1,2\}$) if necessary and applying Lemma \ref{lem1}, we eliminate chart loops except those consisting $h(\sigma_i, e)$.  Thus we have a simplified form.

We consider the case when $\Gamma$ has a positive number of black vertices. By the same process as in the case when $\Gamma$ has no black vertices, we collect white vertices on 1-handles. The middle arcs of a white vertex $w$ on a 1-handle $h$ are connected at both endpoints with $w$; see the rightmost figure of Fig. \ref{fig9}. Hence, if a black vertex is connected with a chart edge on $h$, then, the edge is always a non-middle edge of a white vertex on $h$. Thus, applying CIII-moves if necessary, we assume that there are no black vertices on $h$, and the ends of $h$ are in the neighborhood of each other when we have collected white vertices. Thus, by an addition of 1-handles with chart loops, we eliminate all white vertices and the resulting chart consists of several free edges, chart loops and 1-handles with chart loops. Then, adding $h(\sigma_i, e)$ ($i \in \{1,2\}$) if necessary and applying Lemma \ref{lem1}, we eliminate unnecessary chart loops and we have a simplified form.
\end{proof}

\begin{figure}[ht] 
\includegraphics*[height=6cm]{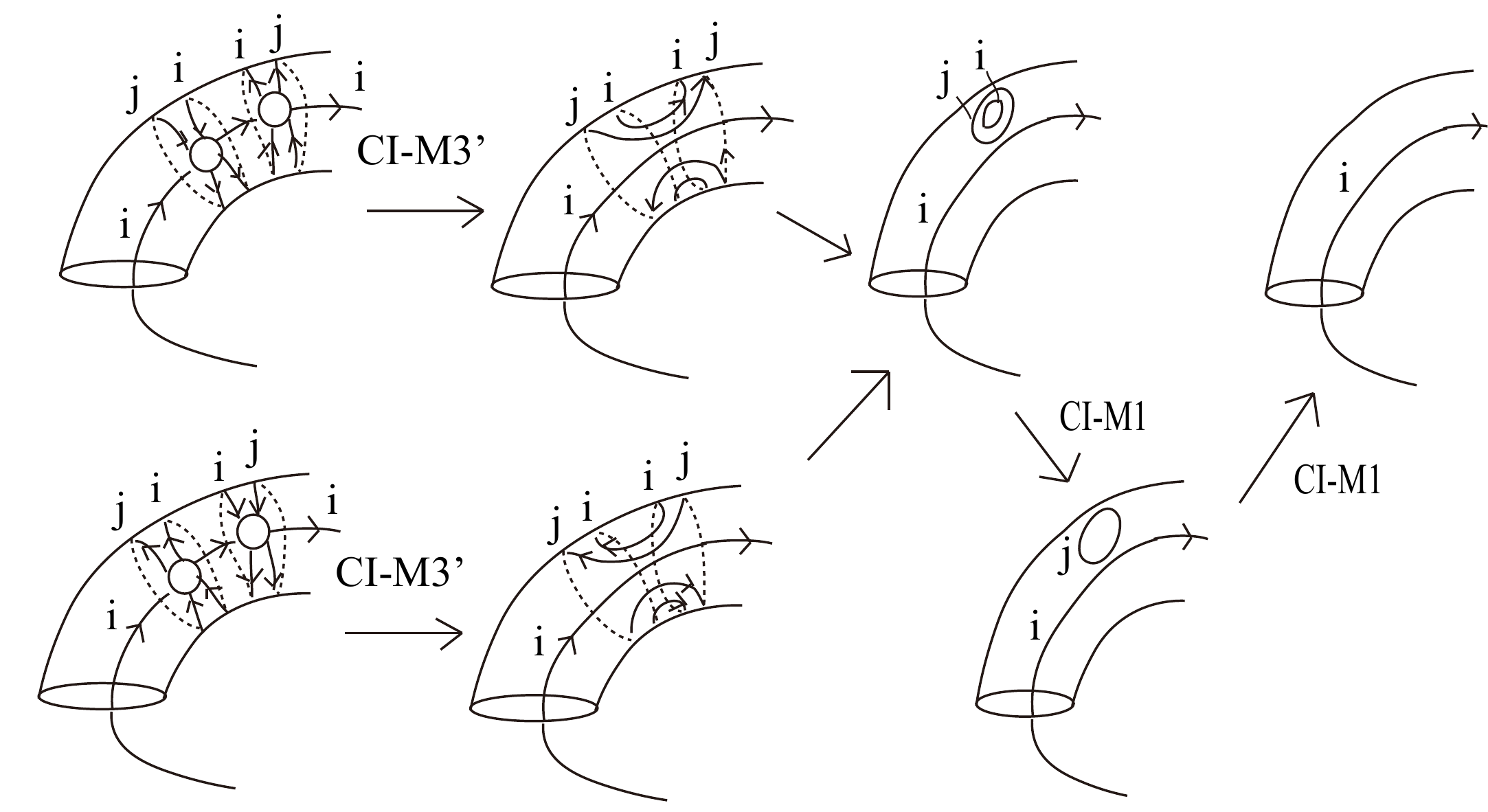}
\caption{Eliminating the pair of white vertices $(w_1, w_2)$, where $\{i,j\}=\{1,2\}$.}
\label{fig12}
\end{figure}

\section{Proofs of Theorems \ref{thm3} and \ref{thm4}}\label{sec4}

\subsection{Key Lemma}\label{sec4-1}
\begin{lemma}\label{lem4}
Let $(F, \Gamma)$ be a branched covering surface-knot of degree $3$. 

$(1)$ If $\Gamma$ is has no black vertices, then, for any white vertex $w$ in $\Gamma$, there is a closed path $\rho \subset F$ with the base point near $w$ such that several white vertices including $w$ can be collected on a 1-handle $h=h(\sigma_i,e)$ for some $i \in \{1,2\}$ by applying a CI-M2 move and sliding one end of $h$ along $\rho$. 

$(2)$ If $\Gamma$ has a black vertex, then $\Gamma$ is equivalent to a chart such that each black vertex is at an endpoint of a free edge. 
\end{lemma}

We prove Lemma \ref{lem4} in Section \ref{subsec5-2}.

\subsection{Notations and Lemmas}\label{sec4-2}

We determine the {\it cocore} of a 1-handle $h=B^2 \times I$ by the oriented closed path $\partial B^2 \times \{0\} \subset h$, with the orientation of $\partial B^2$. 
Further, we determine the base point of the core loop and the cocore of $h$ by their intersection point. We call the ends $B^2 \times \{0\}$ and $B^2 \times \{1\}$ the {\it initial end} and the {\it terminal end}, respectively. We call the arc used to form the core loop of a 1-handle from the core the {\it base arc}. See Fig. \ref{fig13}. 

Let $h$ be a 1-handle whose ends are on a 2-disk $E$. 
For braids $a$ and $b$ which commute, we denote by $h(a,b)$ a 1-handle whose cocore with the given orientation and core loop with the reversed orientation present the braids $a$ and $b$, respectively. We draw such a chart on $h+E$, and this notation $h(a,b)$ is unique \cite{N}. We call $h(a,b)$ a 1-handle with a chart, or simply a {\it 1-handle}. 
We denote by $\sigma_i$ $(i=1,2)$ the standard generator of $B_3$, the braid group of degree 3, and we denote by $e$ the trivial braid in $B_3$. Note that this definition of $h(a,b)$ coincides with that of $h(\sigma_i,e)$ in the previous sections.

\begin{figure}[ht]
\centering
\includegraphics*[height=3.5cm]{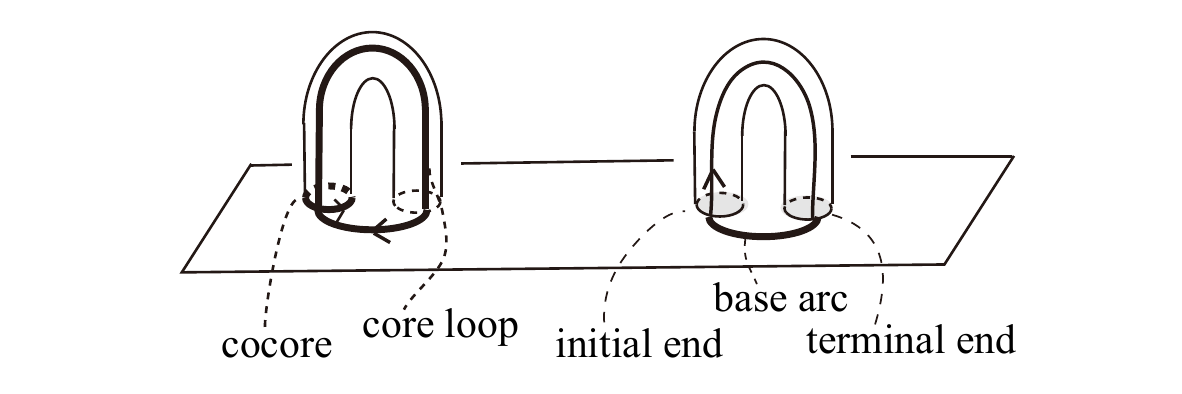}
\caption{The core loop, the cocore, the base arc, the initial end and the terminal end of a 1-handle.}
\label{fig13}
\end{figure}

We say that 1-handles with chart loops are {\it equivalent} if their presenting surfaces are equivalent, and we use the notation \lq\lq $\sim$'' to denote the equivalence relation. When we denote $h(a,b)$ for braids $a, b$, we assume that $a$ and $b$ commute.

We define the type of a white vertex as follows. For a white vertex $w$ such that the three arcs  oriented  toward $w$ consist of two arcs with the label $1$ and one arc with the label $2$ (respectively, two arcs with the label $2$ and one arc with the label $1$), we call $w$ a white vertex of {\it type} $(1,2)$ (respectively, of {\it type} $(2,1)$); see Fig. \ref{fig14}. 

\begin{figure}[ht] 
\includegraphics*[height=3cm]{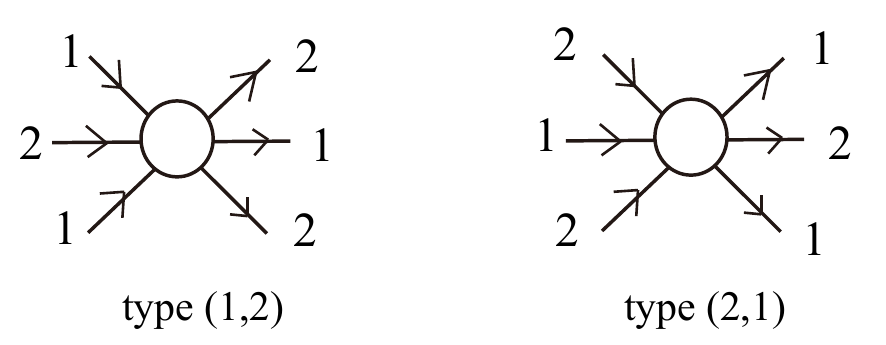}
\caption{A white vertex of type $(1,2)$ and a white vertex of type $(2,1)$.}
\label{fig14}
\end{figure}

In the following lemmas, Lemmas \ref{lem5-2}--\ref{lem7} are also used in \cite{N5, N7,N6}.

\begin{lemma}\label{lem5-2}
Let $(F, \Gamma)$ be a branched covering surface-knots of degree 3.
Let $E_1$ and $E_2$ be disks in $F$ such that $E_1 \cap \Gamma=E_2 \cap \Gamma=\emptyset$. Let $(F_i, \Gamma_i)$  be the result of an addition of $h(\sigma_1, a)+h(\sigma_2, b)$ on $E_i$ $(i=1,2)$. Then, 
$(F_1, \Gamma_1)$ and $(F_2, \Gamma_2)$ are equivalent. 
Thus, by equivalent deformation, $h(\sigma_1, a)+h(\sigma_2, b)$ moves anywhere. 
\end{lemma}

\begin{lemma}\label{lem5-1}
In a chart with no black vertices, the number of white vertices of type $(1,2)$ is the same with that of white vertices of type $(2,1)$. 
\end{lemma}

\begin{lemma}\label{lem3-1a}
Let $\rho$ be a chart loop with the label $1$. 
If there is a 1-handle $h(\sigma_1, b)$ near a neighborhood of an arc of $\rho$, then we can eliminate $\rho$ using $h(\sigma_1,b)$ and a CI-M2 move, where $b$ is a braid. \end{lemma}

The following lemma requires the condition that $h$ is a trivial 1-handle; see Remarks \ref{rem-2} and \ref{rem1}. 

\begin{lemma}\label{lem7}
If $h$ is a trivial 1-handle, then
\[
h(\sigma_2, e) \sim h(e,\sigma_2).
\]
\end{lemma}

When a 1-handle $h$ has collected several white vertices, and the chart arc along the base arc of $h$ has the orientation coherent with the base arc and labeled by $1$, $h$ has the presentation $h(\sigma_1, (\sigma_2 \sigma_1 \sigma_1 \sigma_2)^n)$ for a non-zero integer $n$. The 1-handle $h$ has $2|n|$ white vertices such that the chart loops parallel to the cocore have the coherent orientations as illustrated in Fig. \ref{fig15}. 

\begin{figure}[ht] 
\includegraphics*[height=3cm]{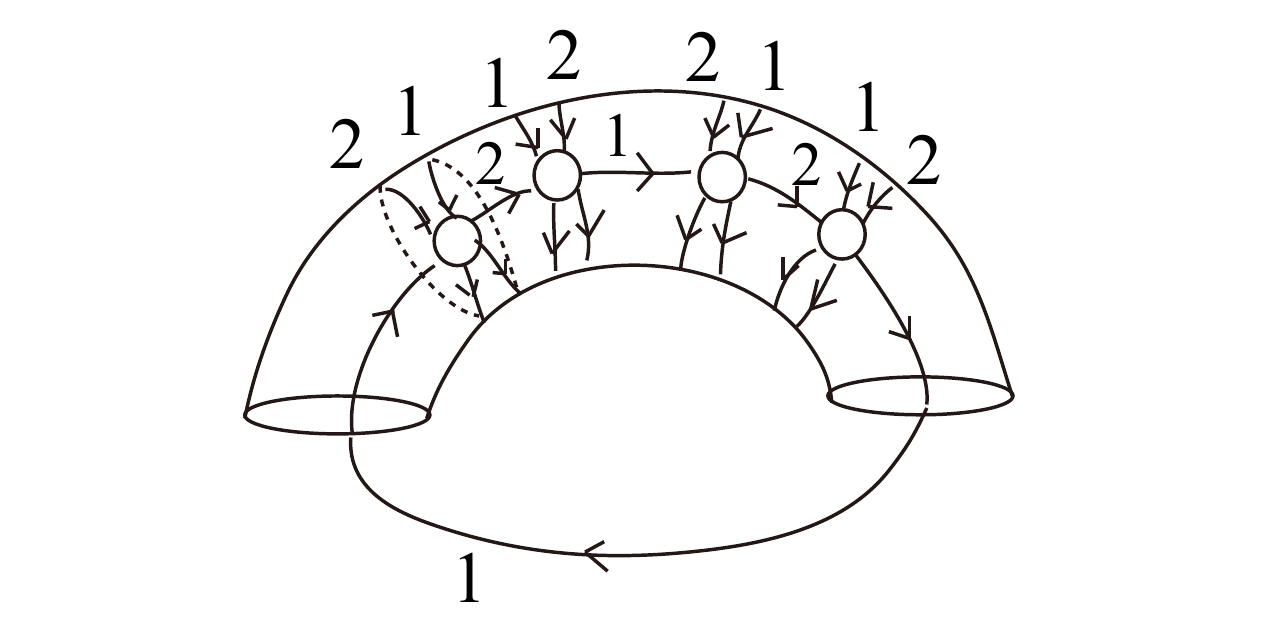}
\caption{A 1-handle $h(\sigma_1, (\sigma_2 \sigma_1 \sigma_1 \sigma_2)^n)$, where $n=2$.}
\label{fig15}
\end{figure}

\begin{lemma}\label{lem5}
Let $n$ be a non-zero integer. 
Let $h=h(\sigma_1, (\sigma_2 \sigma_1 \sigma_1 \sigma_2)^n)$. 

\begin{enumerate}[$(1)$]
\item
The 1-handles $h+h(\sigma_2,e)$ deforms to $h(\sigma_1,(\sigma_2 \sigma_1 \sigma_1 \sigma_2)^{n-1}))+h(e, \sigma_2)$ if $n>0$, or $h(\sigma_1,(\sigma_2 \sigma_1 \sigma_1 \sigma_2)^{n+1}))+h(e, \sigma_2)$ if $n<0$, by equivalent deformation.  

\item
The 1-handles $h+h(e, \sigma_2)$ deforms to $h(\sigma_1,(\sigma_2 \sigma_1 \sigma_1 \sigma_2)^{n-1}))+h(\sigma_2,e)$ if $n>0$, or $h(\sigma_1,(\sigma_2 \sigma_1 \sigma_1 \sigma_2)^{n+1}))+h(\sigma_2,e)$ if $n<0$, by equivalent deformation. 
\end{enumerate}

\end{lemma}

\subsection{Proofs of Theorems \ref{thm3} and \ref{thm4}}\label{sec4-3}

First we show Theorem \ref{thm4}, and then we show Theorem \ref{thm3}.

\begin{proof}[Proof of Theorem \ref{thm4}]
By Lemma \ref{lem5} (2), we have the required result. 
\end{proof}

\begin{proof}[Proof of Theorem \ref{thm3}]
First we show the result for the case when $\Gamma$ has black vertices. 
If $\Gamma$ has a positive number of black vertices, then, by Theorem \ref{thm4}, by equivalence, we deform $\Gamma$ to the form consisting of several free edges and chart loops. We take a free edge $f$, and let $i$ $(i \in \{1,2\})$ be the label of $f$. We add $h=h(\sigma_j, e)$ $(\{j\}=\{1,2\} \backslash \{i\})$ in a neighborhood of $f$. 
By Lemma \ref{lem1}, $h$ eliminates a chart loop of the label $j$ which has an arc in a neighborhood of $h$. Similarly, $f$ eliminates a chart loop of the label $i$ which has an arc in a neighborhood of $f$, by a CI-M2 move and an ambient isotopy as illustrated in Fig. \ref{fig16}. Hence, using $h$ or $f$, we eliminate all chart loops from the one in a neighborhood of $h$ and $f$. The result is free edges and $h$, which is a simplified form. Thus the simplifying number $u(F, \Gamma)$ satisfies $u(F, \Gamma) \leq 1$, which is the required result.  
\begin{figure}[ht]
\centering
\includegraphics*[height=2cm]{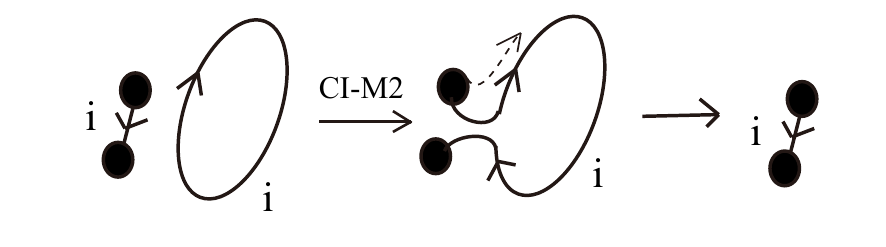}
\caption{Eliminating a chart loop using a free edge.}
\label{fig16}
\end{figure}

We show the result for the case when $\Gamma$ has no black vertices. 
Let $h=h(\sigma_1, e)$ and $h'=h(\sigma_2, e)$. 
We add $h+h'$ in a neighborhood of a non-middle arc of the label 1 of a white vertex $w$ of type $(1,2)$. By Lemma \ref{lem5} (1), there is a closed path with the base point near $w$ such that several white vertices including $w$ are collected on a 1-handle by applying a CI-M2 move and sliding one end of the 1-handle along the path. We collect the white vertices on $h$. Then we have $h=h(\sigma_1, (\sigma_2 \sigma_1 \sigma_1 \sigma_2)^n)$ for some integer $n$. By Lemma \ref{lem5-2}, we move $h+h'$ to a neighborhood of a non-middle arc of the label 1 of a white vertex of type $(1,2)$ of the resulting $\Gamma$, and by the same process, we collect white vertices on $h$. Note that since the number of white vertices of type $(1,2)$ is the same as that of white vertices of type $(2,1)$ by Lemma \ref{lem5-1}, if there are white vertices other than those on $h$, then there exists a white vertex of type $(1,2)$. We repeat this process until we collect all white vertices on $h$. The result is $h(\sigma_1, (\sigma_2 \sigma_1 \sigma_1 \sigma_2)^n)+h(\sigma_2, e)$ for some integer $n$ and several chart loops with the label 1 or 2. By Lemmas \ref{lem1} and  \ref{lem3-1a}, applying CI-M2 moves to the chart loops and the arc of $h(\sigma_1, (\sigma_2 \sigma_1 \sigma_1 \sigma_2)^n)$ or $h(\sigma_2, e)$, we eliminate the loops. Thus we have 
\begin{equation}\label{eq-sigma}
h(\sigma_1, (\sigma_2 \sigma_1 \sigma_1 \sigma_2)^n)+h(\sigma_2, e). 
\end{equation}
If $n$ is even, then, applying Lemma \ref{lem5}, we have $h(\sigma_1, e)+h(\sigma_2, e)$, which is a simplified form. 
If $n$ is odd, then applying Lemma \ref{lem7}, we have
\begin{equation}\label{eq-sigma2}
h(\sigma_1, (\sigma_2 \sigma_1 \sigma_1 \sigma_2)^n)+h(e, \sigma_2);
\end{equation}
see Remark \ref{rem1}. 
 Then, applying Lemma \ref{lem5}, in this case also we have $h(\sigma_1, e)+h(\sigma_2, e)$, which is a simplified form. 
Thus the simplifying number $u(F, \Gamma)$ satisfies $u(F, \Gamma) \leq 2$, which is the required result. 
\end{proof}

\begin{remark}\label{rem1}
In order to consider branched covering surface-knots $(F, \Gamma)$ when $F$ may be knotted, we must be careful to remark when 1-handles are trivial, since Lemma \ref{lem7} requires the condition that the concerning 1-handle is a trivial 1-handle. We first add trivial 1-handles with chart loops, $h(\sigma_1, e)+h(\sigma_2, e)$. When we slide ends of a 1-handle $h=h(\sigma_1, e)$ and collect white vertices, $h$ may become knotted. However, since we don't slide ends of $h(\sigma_2, e)$, $h(\sigma_2, e)$ in (\ref{eq-sigma}) is a trivial 1-handle, and hence we can apply Lemma \ref{lem7} and we obtain (\ref{eq-sigma2}). 
\end{remark}

\section{Proofs of Lemmas \ref{lem4}--\ref{lem5}}\label{sec5}
First we prove Lemmas \ref{lem5-2}--\ref{lem5}, and then we prove Key Lemma \ref{lem4}. 

For a chart edge/arc, we call the endpoint with the orientation from it (respectively, toward it) the {\it initial point} (respectively, the {\it terminal point}). 

\subsection{Proofs of Lemmas \ref{lem5-2}--\ref{lem5}}\label{sec5-1}
Lemmas \ref{lem5-2}--\ref{lem7} are used in \cite{N5, N7, N6}. 
In this section, we give the proofs of Lemmas \ref{lem5-2} and  \ref{lem3-1a} because we use similar methods to prove Key Lemma \ref{lem4}. We also give the proofs of Lemmas \ref{lem5-1} and \ref{lem7} to make this paper self-contained. Figures similar to Figs. \ref{fig17}--\ref{fig20} are in \cite{N5, N7, N6}, especially in \cite{N5, N6}. 
Before the proofs, we give two lemmas which are  also used in \cite{N5, N7, N6}. 
In Section \ref{sec1}, we used $h(\sigma_1, e) \sim h(\sigma_1^{-1}, e)$; see Fig. \ref{fig1}. We have a similar lemma. 

\begin{lemma}\label{lem8}
For a braid $b$, 
\begin{eqnarray}
h(e, b) \sim h(e, b^{-1}).
\end{eqnarray}
\end{lemma}

\begin{proof}
Let $E$ be a 2-disk where the ends of $h(e,b)$ are attached. 
Turning $E$ around and regarding the orientation-reversed core loop as a new core loop, we have $h(e, b^{-1})$: thus $h(e, b) \sim h(e, b^{-1})$; see Fig. \ref{fig17}.  
\end{proof}

\begin{figure}[ht]
\centering
\includegraphics*[height=2cm]{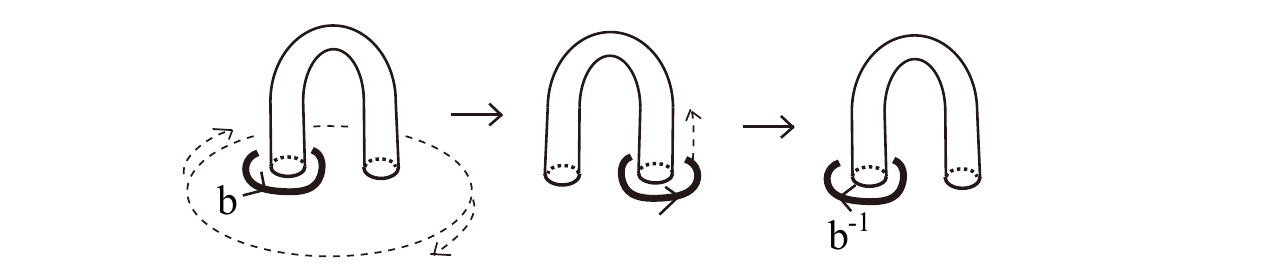}
\caption{The image of $h(e,b) \sim h(e,b^{-1})$, where actually the chart consists of several chart loops along the cocore presenting the braid $b$. }
\label{fig17}
\end{figure}

\begin{lemma}\label{lem9}
For braids $a$, $b$ and $c$,  
\begin{eqnarray}
h(e,c)+h(a,b) & \sim & h(e,bc)+h(a,b) \label{eq5-2}\\
 & \sim & h(e,a^{-1}c)+h(a, b) \label{eq5-3}. 
\end{eqnarray}
\end{lemma}

\begin{proof}
We show (\ref{eq5-2}). Sliding the initial end of $h(e,c)$ along the core of $h(a,b)$ and applying CI-M2 moves as in the first row of Fig. \ref{fig18}, $h(e,c)+h(a,b)$ becomes $h(e, bc)+h(a,b)$.  

We show (\ref{eq5-3}). Moving the initial end of $h(e,c)$ through the chart edges along the base arc of $h(a,b)$ and applying CI-M2 moves, $h(e,c)+h(a,b)$ becomes $h(e, a^{-1}b)+h(a,b)$; see the second row of Fig. \ref{fig18}. In Fig. \ref{fig18}, the crossing change is the deformation such that the crossing of the cores of the 1-handles are changed. Since the 1-handles are in 4-space, this deformation is an equivalent deformation. See \cite[Claim 4.1]{N5} and the proof of \cite[Lemma 4.6]{N5}.
\end{proof}

\begin{figure}[ht]
\centering
\includegraphics*[width=13cm]{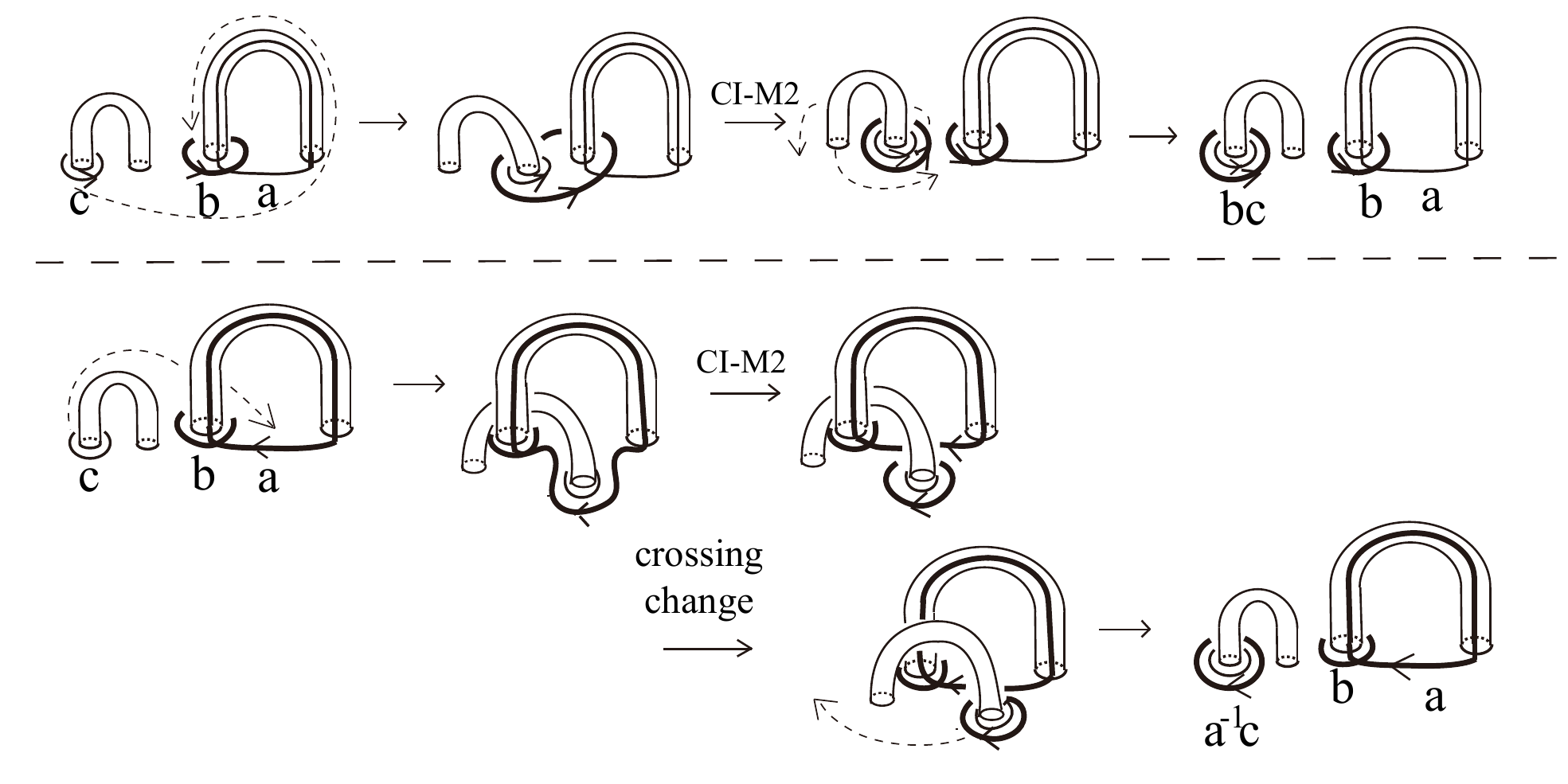}
\caption{The image of $h(e,c)+h(a,b) \sim h(e,bc)+h(a,b)$ (the first row) and $h(e,c)+h(a,b) \sim h(e,a^{-1}c)+h(a,b)$ (the second row). } 
\label{fig18}
\end{figure}

\begin{proof}[Proof of Lemma \ref{lem5-2}]
Let $i$ $(i \in \{1,2\})$ be the label of a chart edge $\rho$ in whose neighborhood we have $h(\sigma_1, a)+h(\sigma_2,b)$. Apply a CI-M2 move between an arc of $\rho$ and an arc along the base arc of $h=h(\sigma_i,c)$ $(c=a$ or $b)$, move the other 1-handle under $h$ to the other side of $\rho$, and by a CI-M2 move again, let $h$ be on the other side of $\rho$; see Fig. \ref{fig19}. Thus $h(\sigma_1, a)+h(\sigma_2,b)$ moves through any chart edge with any label. Hence we have the required result. 
\end{proof}

\begin{figure}[ht]
\centering
\includegraphics*[height=2.5cm]{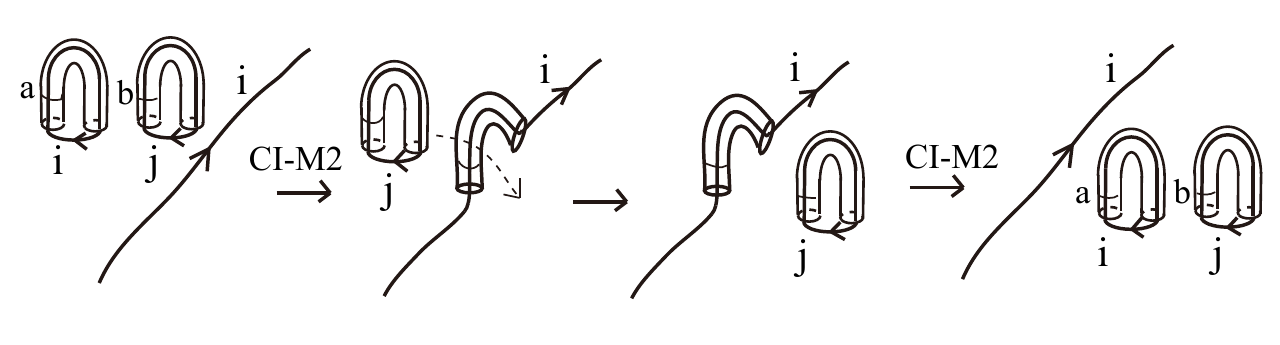}
\caption{The set of 1-handles $h(\sigma_1, a)+h(\sigma_2,b)$ moves anywhere, where $\{i,j\}=\{1,2\}$.}
\label{fig19}
\end{figure}

\begin{proof}[Proof of Lemma \ref{lem5-1}]
A white vertex of type $(1,2)$ (respectively, $(2,1)$) have two arcs (respectively, one arc) with the label 1
 oriented toward the vertex, and one arc (respectively, two arcs) with the label 1 oriented from the vertex. 
Since we have no black vertices, counting the numbers of arcs with the label 1, we see that the numbers of white vertices of type $(1,2)$ is the same with that of white vertices of type $(2,1)$. 
\end{proof}

\begin{proof}[Proof of Lemma \ref{lem3-1a}]
By the same argument as in the proof of Lemma \ref{lem1}, we have the required result. 
\end{proof}

\begin{proof}[Proof of Lemma \ref{lem7}]
Since the 1-handle is trivial, we can regard the core as the cocore, and the cocore as the orientation reversed core; see Fig. \ref{fig25}, see the proof of \cite[Lemma 4.4]{N5} for a more detailed explanation. Thus $h(\sigma_2,  e) \sim h(e, \sigma_2)$, which is the required result.
\end{proof}

\begin{figure}[ht]
\centering
\includegraphics*[height=5cm]{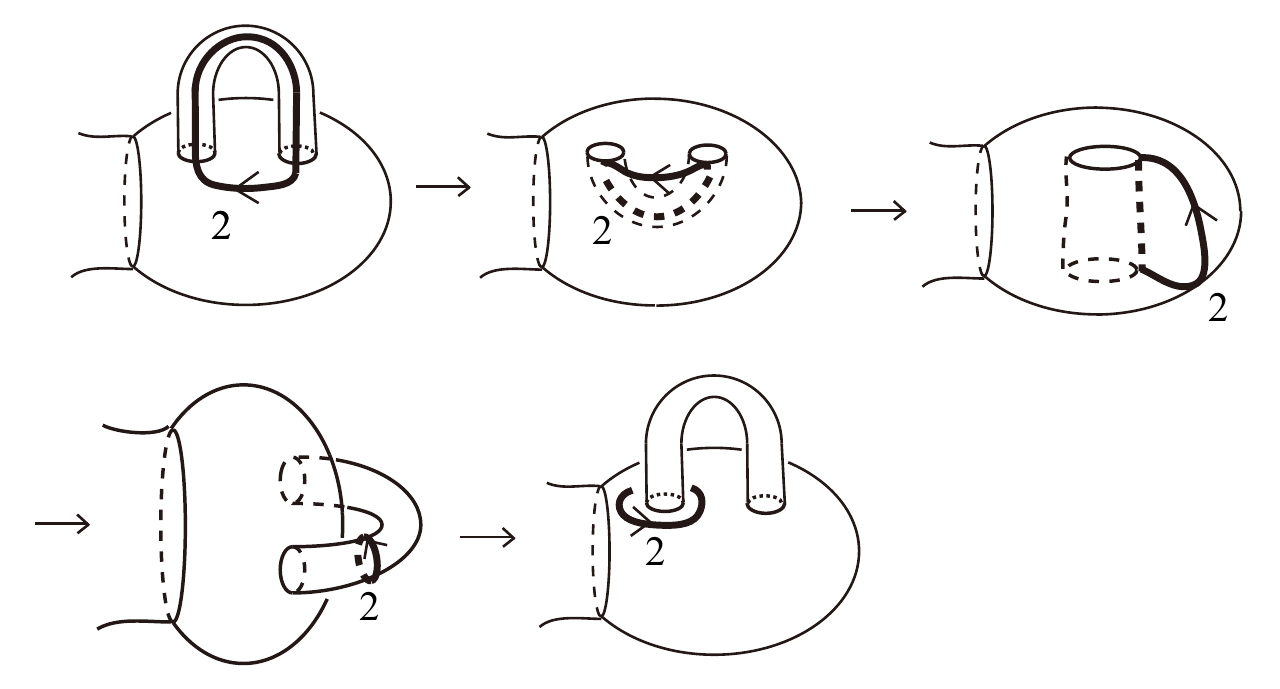}
\caption{For a trivial 1-handle, $h(\sigma_2,e) \sim h(e, \sigma_2)$.}
\label{fig25}
\end{figure}

\begin{proof}[Proof of Lemma \ref{lem5}]
(1) 
We consider the case when $n>0$. The case $n<0$ is shown similarly. Put $h'=h(\sigma_2, e)$. Apply a CI-M2 move to the chart arc of $h'$ and the chart arc of $h=h(\sigma_1, (\sigma_2 \sigma_1 \sigma_1 \sigma_2)^n)$ presenting the first $\sigma_2$ of $(\sigma_2 \sigma_1\sigma_1\sigma_2)^n$, and move the ends of $h'$ and collect all white vertices on $h'$. See Fig. \ref{fig20} for the deformation when we collect the first white vertex on $h'$. See also Lemma \ref{lem3} and Fig. \ref{fig11}. On $h'$, the orientations of the edges parallel to the cocore and around the first white vertex are reversed, and those of the edges around the other white vertices are unchanged. Hence, $h+h'$ deforms to 
\[
h(e, \sigma_1 \sigma_2 (\sigma_2 \sigma_1\sigma_1\sigma_2)^{n-1})+h(\sigma_1, \sigma_2^{-1} \sigma_ 1^{-1} \sigma_1 \sigma_2 (\sigma_2 \sigma_1\sigma_1\sigma_2)^{n-1}).
\]

\begin{figure}[ht]
\centering
\includegraphics*[width=13cm]{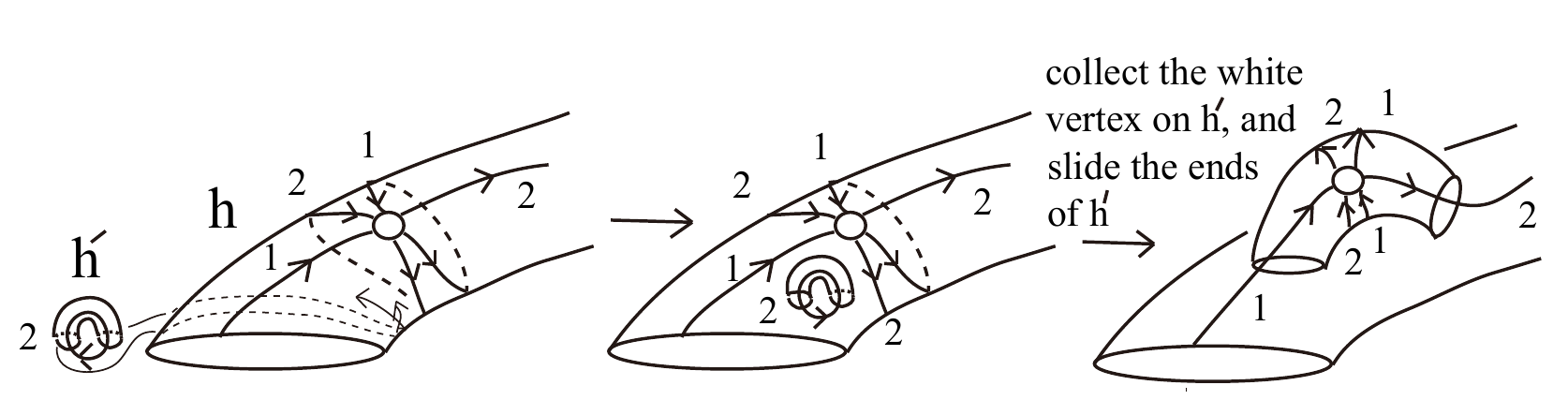}
\caption{Reversal of orientations of the edges around the first  white vertex. }
\label{fig20}
\end{figure}
By CI-M3' and CI-M1 moves as in Fig. \ref{fig12}, or, since $\sigma_2^{-1} \sigma_ 1^{-1} \sigma_1 \sigma_2 (\sigma_2 \sigma_1\sigma_1\sigma_2)^{n-1} \sim  (\sigma_2 \sigma_1\sigma_1\sigma_2)^{n-1}$ and a 1-handle $h(a,b)$ is unique \cite{N}, we have 
\[
h(e, \sigma_1 \sigma_2 (\sigma_2 \sigma_1\sigma_1\sigma_2)^{n-1})+h( \sigma_1, (\sigma_2 \sigma_1\sigma_1\sigma_2)^{n-1}).
\]
By Lemma \ref{lem8}, $h(e, \sigma_1 \sigma_2 (\sigma_2 \sigma_1\sigma_1\sigma_2)^{n-1}) \sim h(e, (\sigma_2 \sigma_1\sigma_1\sigma_2)^{-(n-1)} (\sigma_1 \sigma_2)^{-1})$, and it follows from Lemma \ref{lem9} (\ref{eq5-2}) that we have 
\[
h(e, (\sigma_1 \sigma_2)^{-1})+h(\sigma_1, (\sigma_2 \sigma_1\sigma_1\sigma_2)^{n-1}).
\]
By Lemma \ref{lem8} again, $h(e, (\sigma_1 \sigma_2)^{-1}) \sim h(e, \sigma_1\sigma_2)$; thus we have 
\[
h(e, \sigma_1 \sigma_2)+
h(\sigma_1, (\sigma_2 \sigma_1\sigma_1\sigma_2)^{n-1}).
\]
By Lemma \ref{lem9} (\ref{eq5-3}), we have 
\[
h(e, \sigma_2)+
h(\sigma_1,\ (\sigma_2 \sigma_1\sigma_1\sigma_2)^{n-1}) \sim
h(\sigma_1,\ (\sigma_2 \sigma_1\sigma_1\sigma_2)^{n-1})+h(e, \sigma_2), 
\]
which is the required result.
 
\begin{figure}[ht] 
\includegraphics*[width=13cm]{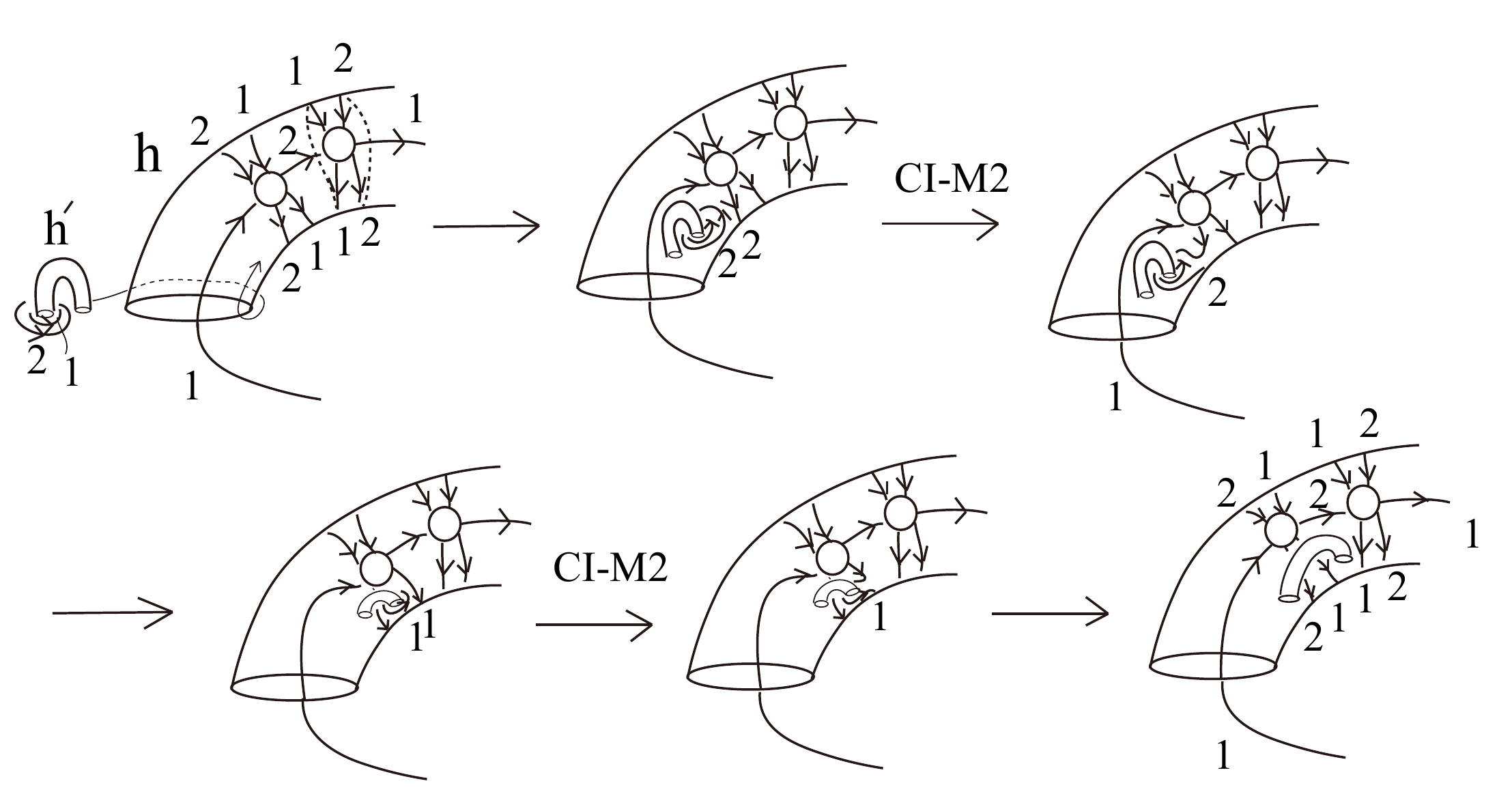}
\caption{Moving the end of the 1-handle $h'$. For simplicity, we omit the labels of several chart edges. }
\label{fig21}
\end{figure}
(2) 
We consider the case when $n>0$. The case $n<0$ is shown similarly. Put $h'=h(e, \sigma_2)$. By Lemma \ref{lem8}, $h' \sim h(e, \sigma_2^{-1})$. Thus $h+h' \sim h'+h$ deforms to  
\[
h(e, \sigma_2^{-1})+h(\sigma_1, (\sigma_2 \sigma_1 \sigma_1 \sigma_2)^n). 
\]
By Lemma \ref{lem9} (\ref{eq5-3}), we have
\[
h(e, \sigma_1^{-1} \sigma_2^{-1})+h(\sigma_1, (\sigma_2 \sigma_1 \sigma_1 \sigma_2)^n). 
\]
 By Lemma \ref{lem8} again, we have 
\[
h(e, \sigma_2 \sigma_1)
+h(\sigma_1, (\sigma_2 \sigma_1 \sigma_1 \sigma_2)^n). 
\]
Applying CI-M2 moves, we deform $h'$ to the form such that $h'$ is equipped with an empty chart and the ends are attached to the places in $h$ as in the last figure of Fig. \ref{fig21}. Let $w_1$ and $w_2$ be the first and the second white vertices associated with the first sequence $\sigma_2 \sigma_1 \sigma_1 \sigma_2$ of $(\sigma_2 \sigma_1 \sigma_1 \sigma_2)^n$. Then, move $w_1$ and $w_2$ on $h'$ and apply a CI-M3'' move to cancel the pair $(w_1, w_2)$. 
Then, on $h'$, there are no white vertices. 

We see the form of the resulting $h'$. 
We denote the initial points and the terminal points of arcs connected to the white vertices $w_1$ and $w_2$ by $p_1, \ldots, p_6$ and $q_1, \ldots, q_6$ as in Fig. \ref{fig22}. See the right figure of Fig. \ref{fig22}. Then, $p_k$ and $q_k$ are connected by a chart arc on $h'$ $(k=1, \ldots,6)$. Further, we take $p_1$ on the initial end of $h$. 
We focus on the edges with the label $1$. We start from $p_1$ and move along the edges. Then we move as follows. We start from $p_1$, we move along the core of $h'$ to $q_1$, along the arc parallel to the cocore of $h$ to $q_5$, along the core of $h'$ to $p_5$, along the arc parallel to the cocore of $h$ to $p_3$, along the core of $h'$ to $q_3$. Thus, the arcs/edges with the label 1 form a connected path $\rho$. Move the terminal end of $h$ along the base arc of $h$ and $\rho$. 
Then, considering a closed path parallel to the cocore of $h$ in a neighborhood of $q_3$ as a new end; see Fig. \ref{fig23}. Then $h$ becomes $h(\sigma_1, (\sigma_2 \sigma_1 \sigma_1 \sigma_2)^{n-1})$. 
The resulting 1-handle $h'$ has only edges with the label 2. 
The 1-handle $h'$ becomes $h(\sigma_2,e)$ by a CI-M1 move and CI-M2 moves as illustrated in Fig. \ref{fig24}. Thus we have the required result.
\end{proof}
 
\begin{figure}[ht] 
\includegraphics*[height=5cm]{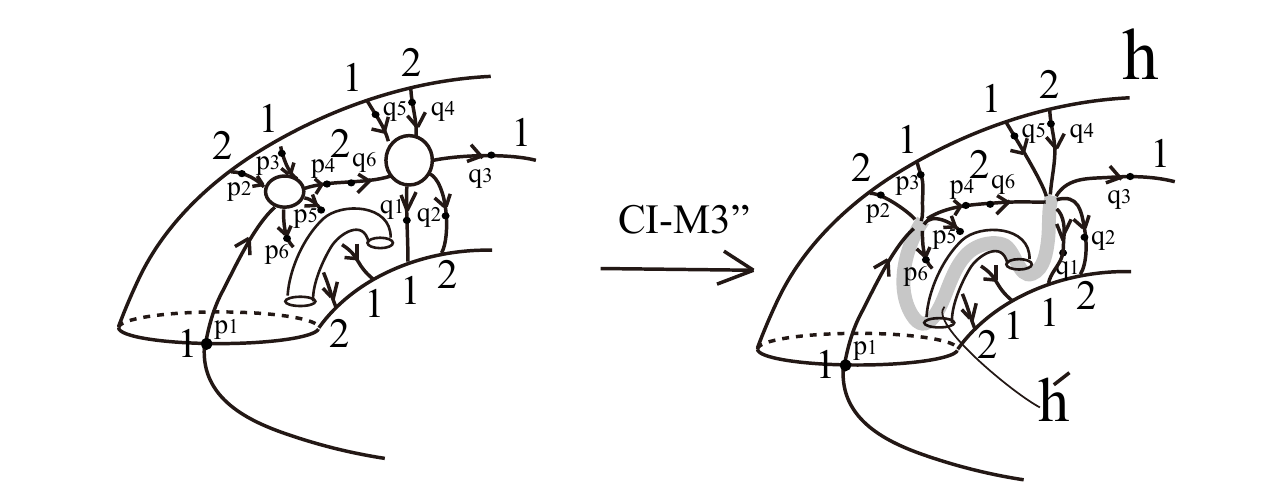}
\caption{Eliminating a pair of white vertices, where in the right figure, we denote the set of parallel chart arcs by a gray ribbon.}
\label{fig22}
\end{figure}

\begin{figure}[ht] 
\includegraphics*[width=13cm]{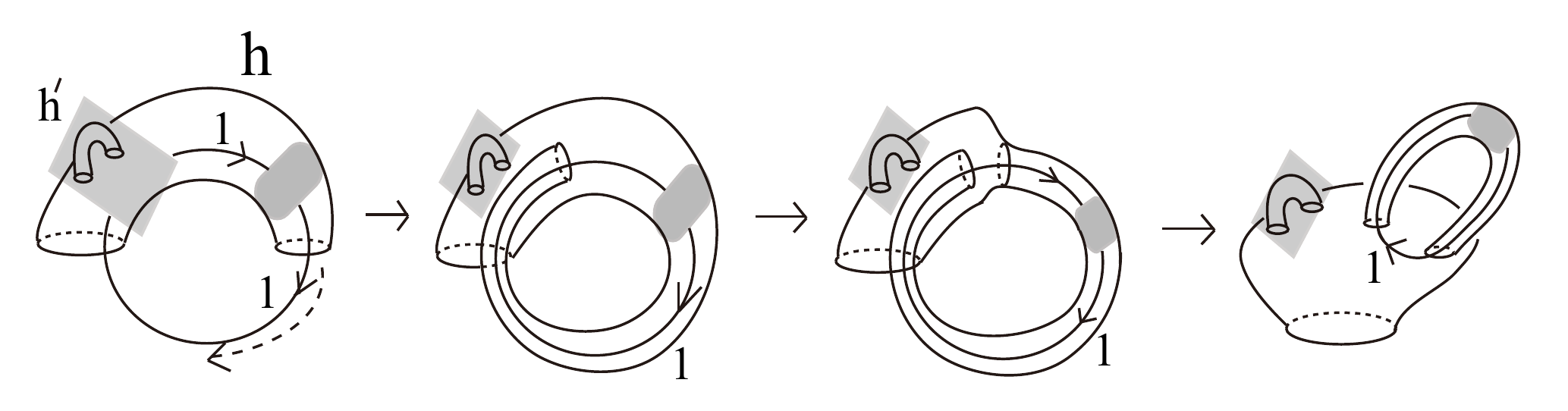}
\caption{Sliding an end of $h$ to obtain a new 1-handle, where there are chart edges and vertices in gray areas.}
\label{fig23}
\end{figure}

\begin{figure}[ht]
\centering
\includegraphics*[width=13cm]{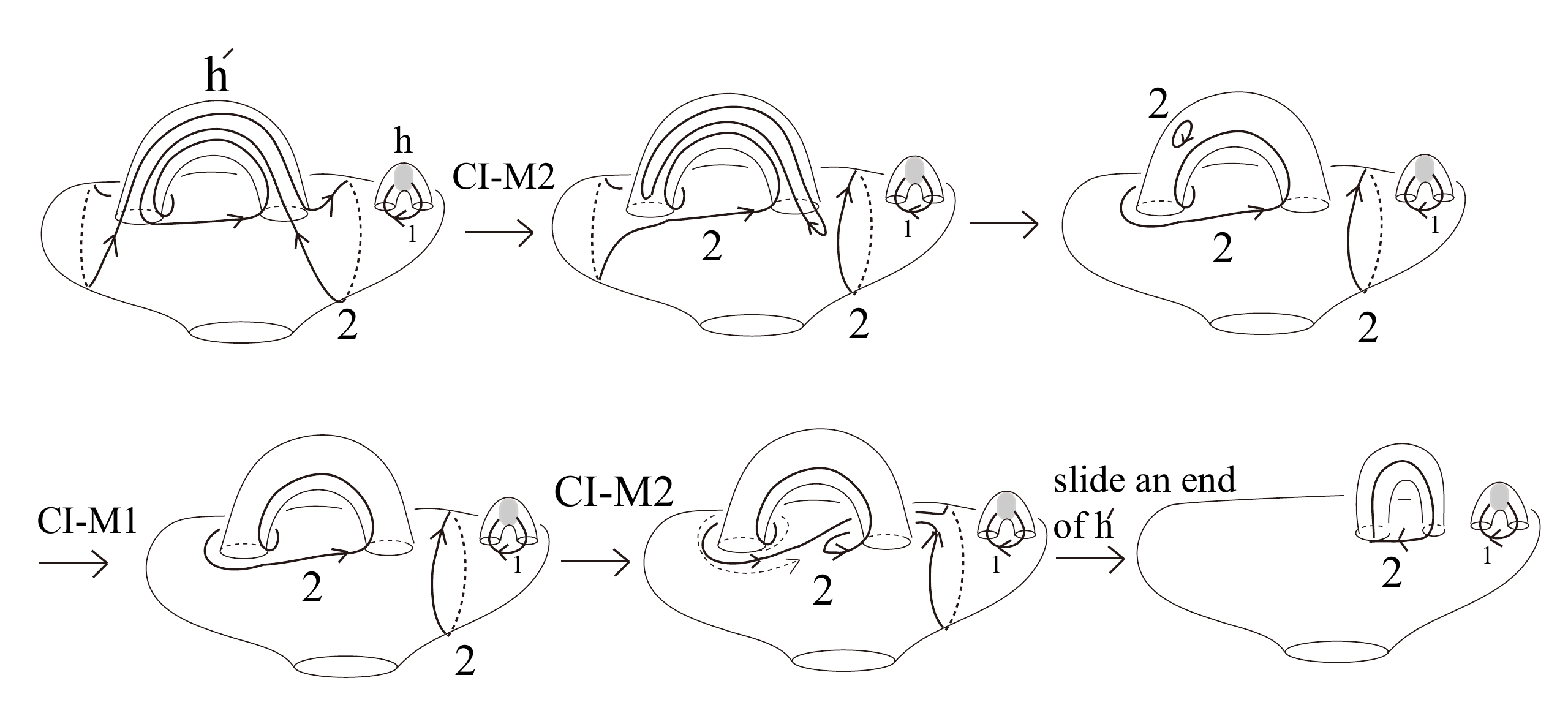}
\caption{The 1-handle $h'$ is equivalent to $h(\sigma_2, e)$. There are chart edges and vertices in gray areas.}
\label{fig24}
\end{figure}

\subsection{Proof of Key Lemma \ref{lem4}}\label{subsec5-2}

\begin{proof}

(1) Let $(F, \Gamma)$ be a branched covering surface-knot of degree 3 such that $\Gamma$ has no black vertices. 
We show that when we attach a 1-handle $h$ to a certain place, we can slide an end of $h$, collecting white vertices including $w$, and the end comes back to a neighborhood of the other end.  

We have a given white vertex $w$. We assume that $w$ is of type $(1,2)$. The other case when $w$ is of type $(2,1)$ is shown similarly. 
We attach a 1-handle $h=h(\sigma_1, e)$ in a neighborhood of a non-middle arc $\alpha_1$ of $w_1=w$ with the label 1. We apply a CI-M2 move and slide an end of $h$ and collect $w$ on $h$. 
Let $\alpha_2$ be the chart edge which comes out from the sliding end of $h$. Then we have two cases: (Case 1) The edge $\alpha_2$ is connected at the terminal point with a white vertex $w_2$ as a non-middle arc, and (Case 2) The edge $\alpha_2$ is connected at the terminal point with a white vertex $w_2$ as a middle arc.

(Case 1) 
Assume that $\alpha_2$ is connected at the terminal point with a white vertex $w_2$ as a non-middle arc. Then, we collect $w_2$ on $h$, and we regard the resulting edge coming out from $h$ (the diagonal edge of $\alpha_2$ with respect to $w_2$) the new $\alpha_2$. 

(Case 2) 
Assume that the chart edge $\alpha_2$ is connected at the terminal point  with a white vertex $w_2$ as a middle arc. Let $i$ $(i\in \{1,2\})$ be the label of $\alpha_2$. Then, $w_2$ is of type $(j,i)$ $(\{j\}=\{1,2\}\backslash \{i\})$. 
By equivalence deformation as in Fig. \ref{fig26}, 
we can connect the end of $h$ with an arc of any edge $\alpha_3$ with the label $i$ which admits a path from $\alpha_2$ to $\alpha_3$ intersecting with no chart edges nor vertices. Since $\alpha_2$ is a middle arc with respect to $w_2$, the two edges sandwiching $\alpha_2$ are distinct edges; thus there exists at least one edge with the label $i$ which admits such a path. 

Let $X$ be the set of edges with the label $i$ such that each element has a path from $\alpha_2$ to itself, intersecting with no chart edges nor vertices. We assume that the elements of $X$ are not contained in the 1-handle $h$. By the above argument, $X$ contains edges other than $\alpha_2$. 

Suppose that $X$ contains an edge $\alpha_3\neq \alpha_2$ which is connected at the terminal point with a white vertex $w_3$ of type $(i,j)$. 
Then, applying the deformation illustrated in Fig. \ref{fig26} to $\alpha_2$ and $\alpha_3$, we connect $w_1$ and $w_3$ by an edge. 
The arcs with the label $i$ of a white vertex of type $(i,j)$ with the orientation toward the vertex are non-middle arcs.  
Hence, we slide the end of $h$, and we collect $w_3$ on $h$. 
 
Suppose that each element of $X$ is an edge connected at the terminal point with a white vertex of type $(j,i)$. 
Since $\alpha_2$ comes out of the sliding end of the 1-handle $h$,
by Lemma \ref{lem5-3}, we take inductively a sequence of edges $\alpha_2, \beta_2, \alpha_3, \beta_3, \alpha_4, \beta_4, \ldots$ where $\alpha_k \in X$ and $\beta_k$ is an edge with the label $j$ $(k=2,3,\ldots)$ such that the closure of the union of the edges forms a connected path $\rho$, and the edges are mutually distinct, until $\alpha_n$ or $\beta_n$ goes into or comes out of the fixed end of $h$ for some $n \geq 3$. We give $\rho$ an orientation induced from that of $\alpha_2$. 
The edge going into or coming out of the fixed end of $h$ is $\alpha_1$, which is an edge with the label $1$ and oriented coherently with that of $\rho$. By Lemma \ref{lem5-3}, $\alpha_k$ is oriented coherently with that of $\rho$, and $\beta_k$ is oriented oppositely to that of $\rho$; thus, $\beta_n$ cannot be $\alpha_1$ and we see that $i=1$ and $\alpha_n=\alpha_1$ for some $n$ $(n\geq 3)$. 
We apply a CI-M2 move between $\alpha_2$ and $\alpha_n$. Then, the sliding end of $h$ is back to a neighborhood of the other end, which is the required result. 

Repeating these processes, we have the required result.

(2) Let $(F, \Gamma)$ be a branched covering surface-knot such that $\Gamma$ has a positive number of black vertices. We take an edge $\alpha_2$ which is connected at the initial point with a black vertex. Remark that counting the numbers of edges connected with black vertices at the initial points and the terminal points, we see that half of the black vertices are connected at the initial points of edges; thus, if there are black vertices, then there exists a black vertex connected at the initial point of an edge. Then, by the same argument as in the proof of (1), 
instead of collecting white vertices, we apply CIII-moves and we move the black vertex, as follows. If $\alpha_2$ is connected at the terminal point with another black vertex, then we have a free edge. If $\alpha_2$ is connected at the terminal point with a white vertex $w_2$ as a non-middle arc, then, applying a CIII-move, we eliminate $w_2$, and we regard the resulting edge connected with the black vertex at the initial point as the new $\alpha_2$. Suppose that $\alpha_2$ is connected at the terminal point with a white vertex $w_2$ as a middle arc. Let $i$ $(i\in \{1,2\})$ be the label of $\alpha_2$. Then, $w_2$ is of type $(j,i)$ $(\{j\}=\{1,2\}\backslash \{i\})$. Because $\alpha_2$ is connected with a black vertex, we cannot take a sequence of edges including $\alpha_2$ such that the closure of the union of the edges forms a closed path; hence, Lemma \ref{lem5-3} implies that there exists an edge $\alpha_3$ with the label $i$ which is connected at the terminal point with a white vertex $w_3$ of type $(i,j)$ or a black vertex. 
Applying a CI-M2 move to $\alpha_2$ and $\alpha_3$, we have one of the two cases: (Case 1) The black vertex and the white vertex $w_3$ of type $(i,j)$ are connected by an edge oriented toward $w_3$ and with the label $i$, or (Case 2) The black vertex is connected with another black vertex by an edge with the label $i$. In Case 1, we eliminate $w_3$ by a CI-M3 move. In Case 2, we have a free edge. Repeating these processes, every black vertex becomes an endpoint of a free edge, and we have the required result. 
\end{proof}

\begin{figure}[ht]
\centering
\includegraphics*[width=12cm]{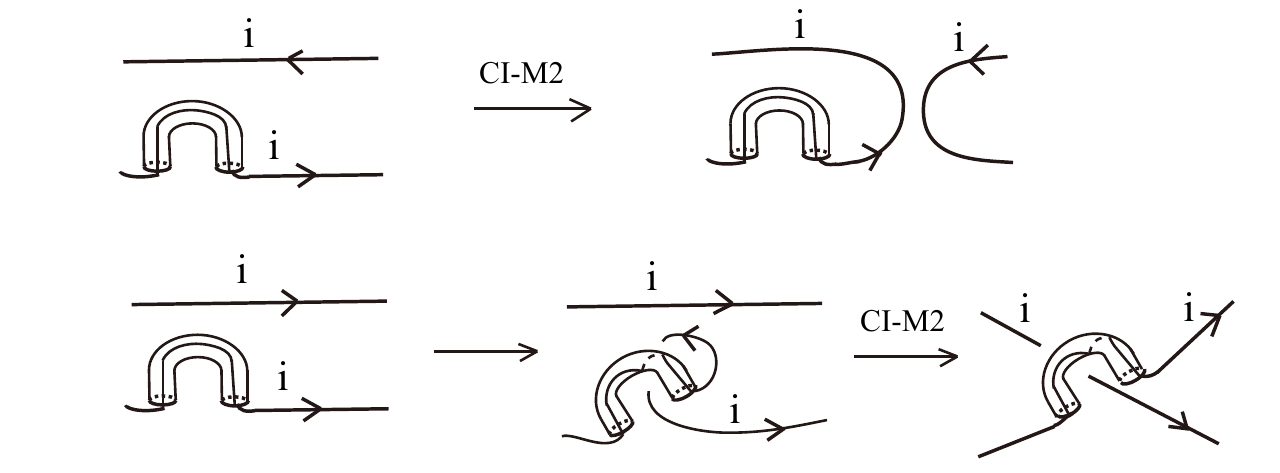}
\caption{An end of a 1-handle with a chart arc along the core,  can connect with another arc with the same label with any orientation. The first row  (respectively, the second row) illustrates the case when the orientations of the arcs are opposite (respectively, coincide).}
\label{fig26}
 \end{figure}

\begin{lemma}\label{lem5-3}
Let $(F, \Gamma)$ be a branched covering surface-knot of degree 3. If necessary, we regard another chart as $\Gamma$ which is C-move equivalent to $\Gamma$ and has a smaller number of white vertices. 
Let $\alpha_2$ be an edge with the label $i$ connected at the terminal point with a white vertex of type $(j,i)$ $(\{i,j\}=\{1,2\})$. 
Let $X$ be the set of edges with the label $i$ such that each element admits a path from $\alpha_2$ to itself, intersecting with no edges nor vertices.  
Assume that each element of $X$ is an edge connected at the terminal point with a white vertex of type $(j,i)$. 
Then, we can take inductively a sequence of edges $\alpha_2, \beta_2, \alpha_3,\beta_3, \alpha_4, \beta_4, \ldots$ (until we have $\alpha_n=\alpha_2$ for some $n\geq 3$) satisfying the following conditions. 

\begin{enumerate}[$(1)$]
\item
The edge $\alpha_k$ is an element of $X$ and $\beta_k$ is an edge with the label $j$ $(k=2,3,\ldots)$. 

\item
The closure of the union of $\alpha_2, \beta_2, \alpha_3,\beta_3, \alpha_4, \beta_4, \ldots$ form a connected path. 

\item
The pair of edges $\alpha_k$ and $\beta_k$ (respectively, $\beta_k$ and $\alpha_{k+1}$) are both connected with the same vertex $w_k$ (respectively, $v_k$) at the terminal points (respectively, the initial points).

\item
The edges $\alpha_2, \beta_2, \alpha_3,\beta_3, \alpha_4, \beta_4, \ldots$ are mutually distinct. 
\end{enumerate}
\end{lemma}

\begin{proof}
We take the sequence of edges $\beta_2, \alpha_3,\beta_3, \alpha_4, \beta_4, \ldots$ 
as follows. First we suppose that each edge is connected with a white vertex at each endpoint. We show later that we can assume this situation. 

Let $w_2$ be the white vertex connected with $\alpha_2$ at the terminal point. Let $\beta_2$ be the adjacent edge of $\alpha_2$ with respect to $w_2$ in the anti-clockwise direction. 
Let $v_2$ be the white vertex at the other endpoint of $\beta_2$. Let $\alpha_3$ be the adjacent edge of $\beta_2$ with respect to $v_2$ in the anti-clockwise direction. By the similar method, we take a sequence of edges $\alpha_2, \beta_2, \alpha_3,\beta_3, \ldots$ $(\alpha_k \in X, \beta_k \in Y, k=2,3,\ldots)$ satisfying the conditions (1) and (2), and a sequence of white vertices $w_2, v_2, w_3, v_3, \ldots$; see Fig. \ref{fig27}, where we don't know the orientations yet. 
 
Now we see the orientations of edges. See Fig. \ref{fig27}. 
By assumption, the edge $\alpha_2$ is with the label $i$ and oriented toward $w_2$. Since $w_2$ is of type $(j,i)$, $\beta_2$ is oriented toward $w_2$. If $\alpha_3$ is oriented toward $v_2$, then $v_2$ is of type $(i,j)$ and $\alpha_3$ is connected with the white vertex $v_2$ of type $(i,j)$ at the terminal point, which implies $\alpha_3 \notin X$, hence a contradiction. Thus we see that $\alpha_3$ is oriented toward $w_3$. Similarly, if $\beta_3$ is oriented toward $v_3$, then $w_3$ is of type $(i,j)$, and $\alpha_3$ is connected with the white vertex $w_3$ of type $(i,j)$ at the terminal point, hence a contradiction. Thus we see that $\beta_3$ is oriented toward $w_3$. By the same argument, we see that $\beta_k$ is an edge with the label $j$ which is oriented from $v_{k}$ to $w_{k}$, and $\alpha_{k+1}$ is an edge with the label $i$ which is oriented from $v_{k}$ to $w_{k+1}$ $(k=2,3, \ldots)$. 
Thus the sequence satisfies the condition (3). 

Since $w_k$ is connected with $\alpha_k$ at the terminal point, by assumption, $w_k$ is of type $(j,i)$ $(k=2,3,\ldots)$. 
We see that $v_k$ is also of type $(j,i)$ $(k=2,3,\ldots)$: If $v_k$ for some $k$ is of type $(i,j)$, then, a by CI-M3' move as in the top right figure of Fig. \ref{fig6}, we eliminate $(v_k, w_k)$, and we regard the resulting chart as the new $\Gamma$. Thus, we have $\Gamma$ such that $w_2, v_2, w_3, v_3, \ldots$ are all of type $(j,i)$. 

Now, we show that each edge of $\beta_2, \alpha_3,\beta_3, \alpha_4, \beta_4, \ldots$ is connected with a white vertex at each endpoint. 
Suppose that $\beta_k$ (respectively, $\alpha_{k+1}$) $(k=2,3,\ldots)$ is connected at the initial point (respectively, the terminal point) with a black vertex. Then, since the white vertices are all of type $(j,i)$, the edge $\beta_k$ (respectively, $\alpha_{k+1}$) is a non-middle edge of the white vertex $w_k$ (respectively, $v_{k}$). Hence, if such a situation occurs, then, by a CIII-move we eliminate the white vertex $w_k$ (respectively, $v_{k}$) and we regard the resulting chart as the new $\Gamma$. Thus we see the existence of the sequence satisfying the conditions (1), (2) and (3). 

Around a white vertex $w_k$ (respectively, $v_{k}$), $\alpha_k$ and $\beta_k$ (respectively, $\beta_k$ and $\alpha_{k+1})$ are both oriented toward $w_k$ (respectively, from $v_{k}$). Hence, if there appears the same white vertex in the sequence more than once (where we exclude the case $\alpha_n=\alpha_2$ for some $n$), then it is only by the form as in Fig. \ref{fig28}, and we see that  the edges $\beta_2, \alpha_3,\beta_3, \alpha_4, \beta_4, \ldots$ are mutually distinct. Thus the sequence satisfies the condition (4), and we have the required result. 
\end{proof}

\begin{figure}[ht]
\centering
\includegraphics*[height=3cm]{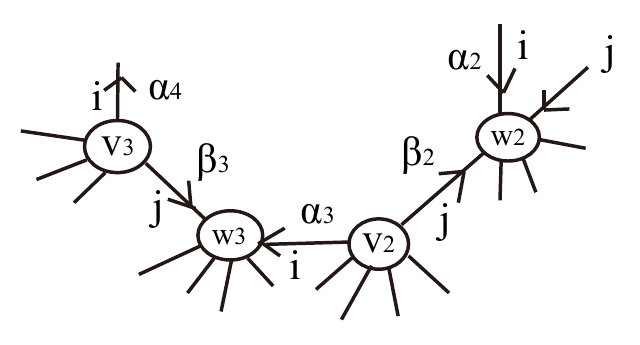}
\caption{The edges $\alpha_2, \beta_2, \alpha_3, \beta_3, \alpha_4$ and the white vertices $w_2, v_2, w_3, v_3$, where $\{i,j\}=\{1,2\}$. We omit the orientations and labels of some edges.}
\label{fig27}
 \end{figure}

\begin{figure}[ht]
\centering
\includegraphics*[height=2.5cm]{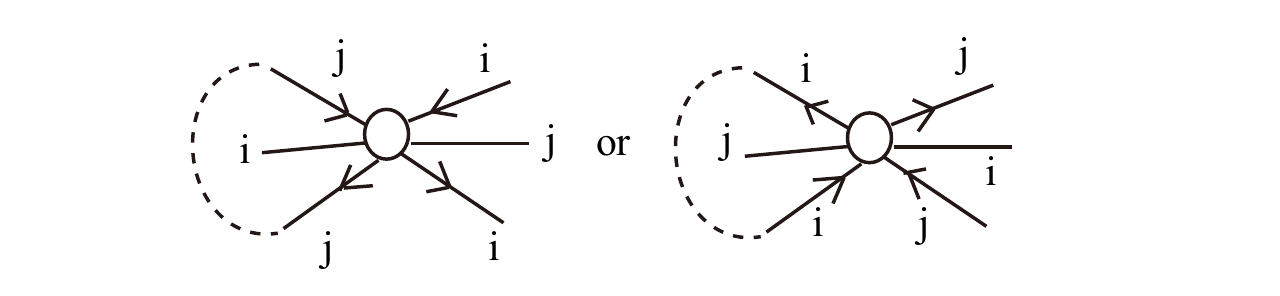}
\caption{The form of a white vertex when it appears more than once in the sequence $w_2, v_2, w_3, v_3, \ldots$, where $\{i,j\}=\{1,2\}$ and we omit orientations of some edges. }
\label{fig28}
\end{figure}

\section*{Acknowledgements}
The author was partially supported by JSPS KAKENHI Grant Number 15K17532.

\end{document}